\documentclass[reqno,12pt]{amsart}
\pdfoutput=1
\usepackage{amsmath,amsthm,amsfonts,amssymb,ifpdf}
\usepackage{amssymb}
\usepackage{amsmath}
\usepackage{amsthm}
\usepackage{thm-restate}
\usepackage{graphicx}
\usepackage[all]{xy}
\usepackage{enumerate}
\usepackage{tikz-cd}
\usepackage{tikz-network}
\usepackage{tikz}
\usetikzlibrary{matrix}
\usepackage{multicol}
\usepackage{enumerate}
\newcommand{\bigzero}{\mbox{\normalfont\Large\bfseries 0}}
\newcommand{\RomanNumeralCaps}[1]
    {\MakeUppercase{\romannumeral #1}}
\theoremstyle{plain} 
\newtheorem{theorem}{Theorem}[section]
\newtheorem{proposition}[theorem]{Proposition}
\newtheorem{lemma}[theorem]{Lemma}
\newtheorem{corollary}[theorem]{Corollary}
\newtheorem{question}{Question}

\theoremstyle{definition} \newtheorem{definition}[theorem]{Definition}
\newtheorem{example}[theorem]{Example}

\theoremstyle{remark} \newtheorem{remark}[theorem]{Remark}

\usepackage[all]{xy}

\ifpdf
  \usepackage[
    pdftex,
    colorlinks,%
    linkcolor=blue,citecolor=red,urlcolor=blue,
    hyperindex,%
    plainpages=false,%
    bookmarksopen,%
    bookmarksnumbered%
  ]{hyperref} 
 
 \usepackage{thumbpdf}
\else
  \usepackage{hyperref}
\fi

\def\deg{{\rm deg}}
\def\Sym{{\rm Sym}}

\def\lim{{\rm lim}}

\def\codim{{\rm codim}}

\def\rc{{\rm rc}}

\newcommand{\ncom}{\newcommand}
\ncom{\mylabel}[1]{{\rm (#1)}\label{#1}}
\ncom{\Hom}{{\textit{Hom}}}
\ncom{\eop}{{\hfill $\Box$}}
\begin{document}
\baselineskip=16pt



\setcounter{tocdepth}{1}

\title[Extending a noncommutative polynomial]{Gram-like matrix preserving extensions and completions of noncommutative polynomials}

\author{Arijit Mukherjee}
\address{Department of Mathematics, Indian Institute of Technology Madras, IIT P.O., Chennai, Tamil Nadu - 600 036, India.}
\email{mukherjee7.arijit@gmail.com}

\author{Arindam Sutradhar}
\address{Department of Mathematics, Indian Institute of Technology Madras, IIT P.O., Chennai, Tamil Nadu - 600 036, India.}
\email{arindam1050@gmail.com}

\begin{abstract}
Given a positive noncommutative polynomial $f$, equivalently a sum of Hermitian squares (SOHS), there exists a positive semidefinite Gram matrix that encrypts all the structural essence of $f$.  There are no available methods for extending a noncommutative polynomial to a SOHS keeping the Gram matrices unperturbed.  As a remedy, we introduce an equally significant notion of Gram-like matrices and provide linear algebraic techniques to get the desired extensions. We further use positive semidefinite completion problem to get SOHS and provide criteria in terms of chordal graphs and 2-regular projective algebraic sets.
\end{abstract}
\maketitle

\tableofcontents

\section{Introduction}
\label{sec:Introduction}
Hilbert's 17th problem states whether every multivariate polynomial that takes only nonnegative values over the reals can be represented as sum of squares (SOS) of polynomials. Hilbert, in 1888, proved that it is true only for quadratic forms, binary forms and ternary quatrics in the commutative case. In the noncommutative setup, Helton proved that any polynomial is positive if and only if it is a sum of Hermitian squares (SOHS). The SOS and SOHS techniques are heavily used for polynomial optimization and convex optimization, in particular, in semidefinite programming (SDP). Many of the SOHS techniques are taken from real algebraic geometry and inspired from functional analysis.

The SOHS techniques apply across many different disciplines including control theory, probability and statistics. Out of many directions, hierarchy techniques are still very active areas of research in quantum information theory (cf. \cite{FangHamzi} and \cite{NavStePir}) and have been proven useful in discriminating the quantum states. A large portion of the theory for noncommutative polynomials is being developed by Helton, McCullough, Klep and their collaborators (cf. \cite{KlepCafPo1}, \cite{HeMc}, \cite{HeMcPu},  \cite{KlepPo} and \cite{KlSc}). The book ``Optimization of Polynomials in Non-commuting Variables'' (cf. \cite{BKP}) by Burgdorf, Klep and Povh is a beautiful introduction to the subject. Pablo Parillo has used SOS techniques in different SDP to produce different control theory oriented applications (cf. \cite{Parrilo}). The work of Blekherman and his collaborators on SOS techniques provide fruitful insight about the fundamental geometric properties of some projective algebraic sets (cf. \cite{BSV1}, \cite{BSV2} and \cite{BSSV}). Motivated by the theory of noncommutative polynomials, noncommutative functions have been developed over noncommutative sets and many researchers are working on the subject (cf. \cite{AglerMac} and \cite{Vinnikov}).

Our main contributions are threefold. First, we provide a linear algebraic construction to extend noncommutative polynomials to SOHS, by adding necessary terms, that ensures the Gram-like matrix of the SOHS part is a principal submatrix of the extension’s matrix (cf. Theorem \ref{Block PSD technique_with non-SOHS part having multiple terms}). Second, we leverage chordal graph theory to characterize when a partially specified noncommutative polynomial can be completed to a SOHS, linking matrix completion to SOHS completion (cf. Theorem \ref{main theorem_SOHS completion of a quasi-SOHS iff known pattern is chordal}). Third, we connect these completions to projective algebraic sets, showing that their 2-regularity guarantees SOHS completions (cf. Theorem \ref{main theorem_SOHS completion of a quasi-SOHS iff unknown pattern is 2-regular}). These results blend linear algebra, graph theory, and algebraic geometry.  One of the key aspects that the reader will take away from this paper is some techniques to produce ample SOHS.  Moreover, these techniques will be fruitful in polynomial optimization.


We now go through the chronology of the paper in a bit more detail.  In Section \ref{sec: Prerequisites}, we recall some prerequisites and discuss the notion of a Gram matrix \& monomial vector associated to a symmetric polynomial \& SOHS. We introduce the notion of a Gram-like matrix which reduces the size of Gram matrix and disregards the order of the constituent monomials in the monomial vector  (cf. Remark \ref{remark_comparing Gram matrix and Gram like matrix} for a detailed comparison between a Gram matrix and a Gram-like matrix).  Moreover, it keeps the soul of a Gram matrix intact.  

The following question captures the central theme of this paper:
\begin{question}\label{Introduction_main question_possible modification}
 Given a noncommutative polynomial $f \in \mathbb{R}\langle\underline{X}\rangle$ that includes SOHS components, can we construct a SOHS polynomial $ \widetilde{f} \in \mathbb{R}\langle\underline{X}\rangle $ that extends (or completes) $f$ such that the positive semidefinite Gram-like matrices of the SOHS components are principal submatrices of the positive semidefinite Gram-like matrix of $\tilde{f}$?   
\end{question}
Theorem \ref{Block PSD technique_with non-SOHS part having multiple terms}, Theorem \ref{main theorem_SOHS completion of a quasi-SOHS iff known pattern is chordal} and Theorem \ref{main theorem_SOHS completion of a quasi-SOHS iff unknown pattern is 2-regular} answer Question \ref{Introduction_main question_possible modification} in three different set up in three consecutive subsections.

\begin{enumerate}
\item 
In subsection \ref{subsec: Block Matrix Technique}, assuming that there exists some $\widetilde{f}$ satisfying the properties as in Question \ref{Introduction_main question_possible modification}, we first note that the structure of a positive semidefinite Gram-like matrix $G_{\widetilde{f}}$ of such $\widetilde{f}$ would be of the form $\begin{pmatrix}
    G_h & A\\
    A^t & B
\end{pmatrix}$, where $G_h$ is some positive semidefinite Gram-like matrix of $h$ (cf. Lemma \ref{Lemma_form of Gram-like matrix matrix of Gram-like matrix preserving extension}).  {\color{blue} Then we provide a way to produce a positive semidefinite matrix out of a partially specified symmetric matrix of the aforementioned structure only by controlling the diagonal elements of $B$ (cf. Lemma \ref{Lemma_generating a psd matrix by choosing the diagonals of the lower diagonal block}.}  Exploiting these two results, we further provide some necessary and sufficient conditions in terms of a representation of a SOHS part $h$ of $f$ and monomial decompositions of $f-h$ (cf. Theorem \ref{Theorem_conditions satisfied my monomial vectors if if G-l MPE exists}).      



We provide necessary and sufficient condition for $f=h+\sum_j d_j \delta_j$,
with $h$ SOHS and $\delta_j$ are not hermitian squares, to have a SOHS extension $\tilde{f}$. If for any decomposition of $\delta_j=l_kr_k$, $l_k$ and $r_k$ simultaneously do not belong to the set of monomial vectors of a Gram-like matrix $G_h$ and if $G_h$ is positive definite, we can always extend $f$ to a Gram-like matrix preserving $\tilde{f}$. But if $G_h$ is psd but not pd, we need to have additional conditions, namely, the column space of $A$ should be sitting inside the column space of $G_h$. We see that the problem boils down to choosing diagonal elements of lower diagonal block of a $2 \times 2$ block matrix, so that we can make the block matrix positive semidefinite.

This, in turn, helps us to find noncommutative polynomials $f$ that do not admit any such extension $\widetilde{f}$ (cf. Example \ref{Examples_polynomials that do not admit G-L MPE}).  We then provide necessary and sufficient conditions for the existence of such $\widetilde{f}$, when $G_{\widetilde{f}}$ takes a special form, namely as in \eqref{equation_structure of a Gram-like matrix preserving Gram-like matrix} (cf. Theorem \ref{Theorem_NASC for block matrix form of a Gram-matrix of the extension}).  This characterisation helps us to produce a plenty of SOHS polynomials such that their difference is also a SOHS, which is not true in general (cf. Remark \ref{Remark_producing examples of sohs-sohs=sohs}).       

Apart from these existential criteria of such $\widetilde{f}$ in Theorem \ref{Theorem_conditions satisfied my monomial vectors if if G-l MPE exists} and Theorem \ref{Theorem_NASC for block matrix form of a Gram-matrix of the extension}, we also provide a rigorous construction of $G_{\widetilde{f}}$ (and hence $\widetilde{f}$) for a class of noncommutative polynomials $f=h+\sum_{j=1}^ra_j\zeta_j$, $h$ being SOHS.  We do so whenever $\rc(\eta,i)\neq \zeta_j$, $1\leq j \leq r$, or equivalently, whenever the monomials $\zeta_j$ are not equal to any right end part of any monomials $\eta$ of $h$ (cf. Theorem \ref{Block PSD technique_with non-SOHS part having multiple terms}).



\item In subsection \ref{subsec : Positive semidefinite completion problem and SOHS via graphs}, we deal with Question \ref{Introduction_main question_possible modification} via graph theory.  In fact, we look at positive semidefinite completion problem in the context of completion of a quasi SOHS polynomial (cf. Definition  \ref{Definition_quasi SOHS} for details).  We prove the following theorem which describes a pattern made out of the known coefficients of a quasi SOHS polynomial that ensures a SOHS completion of the same (cf. Theorem \ref{main theorem_SOHS completion of a quasi-SOHS iff known pattern is chordal}).  The pattern is given in terms of a graph associated to a partial Gram-like matrix of the given quasi SOHS polynomial.  To be specific, we prove the following : 

Let $f\in \Sym\;\mathbb{R}\langle \underline{X}\rangle$ be a quasi SOHS polynomial w.r.t a representation $\mathbf{R}_f$, as in
\eqref{equation_ most general form of a representation of a symmetric polynomial with positive co-efficients of square terms}, along with conditions as in \eqref{conditions on the coefficients of general form of a representation of a symmetric polynomial with positive co-efficients of square terms}.  Then any such $f$ has a SOHS completion $\overline{f}^{\;G_f,\;W_k}$ w.r.t the partial positive semidefinite Gram-like matrix $G_f$ associated with $\mathbf{R}_f$ if and only if the specification graph of $G_f$ is chordal.  Here $W_k$ is the column vector $\begin{pmatrix}
\zeta_1 & \cdots & \zeta_k
\end{pmatrix}^{t}$, $\zeta_i$, $1\leq i \leq k$, being the monomials appearing in $\mathbf{R}_f$.

Here a positive semidefinite Gram-like matrix of any such SOHS completion keeps the positive semidefinite Gram-like matrices of SOHS parts of the given quasi SOHS polynomial as its principal submatrices.  We also prove a similar statement for Gram matrix as well, (cf. $(1)$ of Theorem \ref{main theorem_SOHS completion of a quasi-SOHS iff known pattern is chordal}).  

\item In subsection \ref{subsec : Positive semidefinite completion problem and SOHS via varieties}, we again answer when a quasi SOHS polynomial can be completed to a SOHS but now the pattern of known and unknown entries is guided by a projective algebraic set.  Moreover, the defining equations of this algebraic set are homogeneous square-free quadratic (commutative) monomials.  The nobleness of this elucidation is that a SOHS completion of some partial symmetric noncommutative polynomial of any even degree is controlled by some homogeneous (commutative) monomials of degree $2$.  To be more specific, we prove that $2$-regularity of such a projective algebraic set guarantees the SOHS completion.  Furthermore, it is deeply related to Theorem \ref{main theorem_SOHS completion of a quasi-SOHS iff known pattern is chordal} as well.  Precisely, we obtain the following (cf. Theorem \ref{main theorem_SOHS completion of a quasi-SOHS iff unknown pattern is 2-regular} for details):

Let $f\in \Sym\;\mathbb{R}\langle \underline{X}\rangle$ be a quasi SOHS polynomial w.r.t a representation $\mathbf{R}_f$, as in
\eqref{equation_ most general form of a representation of a symmetric polynomial with positive co-efficients of square terms}, along with conditions as in \eqref{conditions on the coefficients of general form of a representation of a symmetric polynomial with positive co-efficients of square terms}.  Then any such $f$ has a SOHS completion $\overline{f}^{\;G_f,\;W_k}$ w.r.t the partial positive semidefinite Gram-like matrix $G_f$ associated with $\mathbf{R}_f$ if and only if the subspace arrangement $V_{G_f}(\subseteq \mathbb{P}^{k-1})$ of $G_f$ is $2$-regular.

We provide some examples to display simultaneous applications of Theorem \ref{main theorem_SOHS completion of a quasi-SOHS iff known pattern is chordal} and Theorem \ref{main theorem_SOHS completion of a quasi-SOHS iff unknown pattern is 2-regular}. In these examples we use Macaulay2 software to calculate regularity via Betti tables and free resolutions.
\end{enumerate}

\section{Background material}
\label{sec: Prerequisites}
We introduce our notations and recall necessary definitions and results in this section.
\subsection{Noncommutative polynomial and SOHS}
\label{subsec: Non-commutative polynomial and SOHS}
Let $X_1,X_2,\dots,X_n$ be noncommuting variables. Generating all possible \textit{words}, also referred as \textit{monomials}, in these letters with finite length by concatenating the letters and adding $1$, the empty word, we get a free monoid. Let $\langle \underline{X} \rangle=\langle X_1,X_2,\dots,X_n \rangle$ be the free monoid generated by $X_1, X_2, \dots, X_n$. We take all possible finite linear combinations of words from $\langle \underline{X} \rangle$ with coefficients from $\mathbb{R}$ and we call them \textit{real noncommutative polynomials}. Denote the set of all real noncommutative polynomials as
\[\mathbb{R} \langle \underline{X} \rangle=\mathbb{R} \langle X_1,X_2,\dots,X_n \rangle= \Big\{\sum_{i=1}^{N} a_iW_i\mid \; a_i \in \mathbb{R}, N \in \mathbb{N}, W_i \in \langle \underline{X} \rangle \Big\}.\]
$\mathbb{R} \langle \underline{X} \rangle$ forms a free algebra. We equip $\mathbb{R} \langle \underline{X} \rangle$ with the $*$ operation
\begin{align*}
    \mathbb{R}\langle \underline{X} \rangle &\longrightarrow \mathbb{R} \langle \underline{X} \rangle\\
    f &\longmapsto f^*.
\end{align*}
Under this map, $\mathbb{R} \cup \{\underline{X}\}$ gets fixed and for given
\begin{align*}
    W&=X_1X_2\dots X_{n-1}X_n,\\
    W^*&=X_nX_{n-1}\dots X_2X_1.
\end{align*}
By the \textit{degree of a monomial} $W$, denoted by $\deg(W)$, we mean the length of $W$.  For example, $\deg(X_1X_2^3X_3)=\deg(X_1X_2X_2X_2X_3)=5$.  By the \textit{degree of a noncommutative polynomial} $f$, denoted by $\deg(f)$, we mean the degree of any one of its highest degree monomials. 

We reserve the notation $\Sym\; \mathbb{R} \langle \underline{X} \rangle$ to denote the set of all \textit{symmetric polynomials} of $\mathbb{R} \langle \underline{X} \rangle$, that is, 
\[\Sym\; \mathbb{R} \langle \underline{X} \rangle =\{f \in \mathbb{R} \langle \underline{X} \rangle \mid \; f=f^* \}.\]

Let $\langle \underline{X} \rangle_d$ be the set of all words in $\langle \underline{X} \rangle$ of at most degree $d$. Let $\mathbb{R}\langle \underline{X} \rangle_d$ be the set of all polynomials with monomials coming from $\langle \underline{X} \rangle_d$. The cardinality of  $\langle \underline{X} \rangle_d$ is $s(d,n)=\sum_{i=0}^d n^i$, where $n$ is the number of noncommuting variables in $\langle \underline{X} \rangle$. Let $W_{s(d,n)}$ be the $s(d,n) \times 1$ column vector of the words in $\langle \underline{X} \rangle_d$ taken in the lexicographic order. Given a polynomial ${f} \in \mathbb{R}\langle X \rangle_d$, $\mathbf{f} \Big(=({f_w})_{w \in \mathbb{R}\langle X \rangle_d} \Big)$ be its coefficient vector. Every 
polynomial ${f}$ of degree $d$ can be expressed as $f=\sum_{w \in \langle \underline{X} \rangle_d} f_w w={\mathbf{f}}^T W_{s(d,n)}=W_{s(d,n)}^*\mathbf{f}$.

\begin{definition}
    A noncommunicative polynomial $f \in \mathbb{R} \langle X \rangle$ is \textit{positive} if $f(A_1,\dots,A_n)$ is positive semidefinite for all tuples $(A_1, \dots,A_n)$ of real symmetric matrices of the same order.
\end{definition}

\begin{definition}
For any $p \in \mathbb{R}\langle \underline{X} \rangle$, the noncommutative polynomial $p^*p$ is called a \textit{Hermitian square}. A noncommutative polynomial $f \in \mathbb{R}\langle \underline{X} \rangle$ is said to be \textit{a sum of Hermitian squares} (or \textit{SOHS} in short) if it can be written as a sum of some Hermitian squares, i.e. there exists noncommutative polynomials $g_1,g_2,\dots,g_k \in \mathbb{R}\langle X \rangle$ for some $k \in \mathbb{N}$ such that 
\[f=g_1^*g_1+g_2^*g_2+\dots+g_k^*g_k.\]
\end{definition}

\begin{remark} \label{rem_Heltonresult}
A noncommutaive polynomial is positive if and only if it is SOHS. See \cite{H} for Helton's proof of the result and see \cite{McPu} for McCullough and Putinar's alternative proof.
\end{remark}

We now recall a notion that formalizes any right end part of a given monomial (cf. \cite[Section 3, p.~745]{KlepPo} or \cite[Definition 2.5, p.~38]{BKP}).
\begin{definition}\label{definition_right chip function}
    The \textit{right chip function} $\rc:\langle \underline{X} \rangle \times \mathbb{N}\rightarrow \langle \underline{X} \rangle$ is defined by 
\begin{equation*}
\rc(W_1W_2\cdots W_n,i)=\left\{ \begin{array}{ll}
W_{n-i+1}W_{n-i+2}\cdots W_n & \mbox{if $1\leq i < n$};\\\\
W_1W_2\cdots W_n & \mbox{if $i \geq n$};\\\\
1 & \mbox{if $i=0$}.\end{array} \right.
\end{equation*}
Here $\deg(W_i)=1$ for all $1\leq i \leq n$.
\end{definition}
\begin{example}
    Let $W=X_1^2X_2^2X_3X_2^2X_1$ be a noncommutative word. Then, $\rc(W,4)=X_3X_2^2X_1$, $\rc(W,8)=\rc(W,9)=X_1^2X_2^2X_3X_2^2X_1$ and $\rc(W,0)=1$.
\end{example}

\subsection{Some properties of positive definite and semidefinite matrices}
\label{subsec: Basic properties of positive definite and semidefinite matrices}


We now recall the characterisations of positive definite and positive semidefinite matrices in terms of their principal minors.
\begin{lemma}[Sylvester's criterion for positive definite matrices]\label{Sylvester's criterion for pd matrix}
    A real symmetric matrix is positive definite if all its leading principal minors are positive.
\end{lemma}
\begin{proof}
See \cite[\S 3.3, 3.15, p. 27]{Ba1}.
\end{proof}
\begin{lemma}[Generalized Sylvester's criterion for positive semidefinite matrices]\label{Generalized Sylvester's criterion for psd matrix}
    A real symmetric matrix is positive semidefinite if all its principal minors are nonnegative.
\end{lemma}
\begin{proof}
See \cite[\S 1.2, p. 7]{Bh}.
\end{proof}
The following well-known results are easy consequences of spectral decomposition of a real symmetric matrix and will be used often in the proofs of the results in Section \ref{sec: how to add new monomials to get a SOHS}. The following results are direct consequences of the spectral theorem for real symmetric matrices.

\begin{lemma}\label{psdness is prserved by same set of row and column interchange}
Let $G$ be a positive semidefinite matrix with real entries of size $k$.  Then the matrix obtained from $G$ by first interchanging its $i$-th and $j$-th row (column respectively) followed by interchanging its $i$-th and $j$-th column (row respectively) is again a positive semidefinite matrix, for any $1\leq i,j\leq k$.  
\end{lemma}
\begin{lemma}\label{Lemma_block psd(pd)implies psd(pd)}
Let $A$ be a block diagonal matrix with $A_1, A_2,\ldots, A_n$ being its diagonal blocks.  If each $A_i$ is a positive definite (respectively positive semidefinite) matrix, then so is $A$.
\end{lemma}
We recall another well-known property of a positive semidefinite matrix. We present a rudimentary argument, which is available in the literature, for the reader not familiar with this result. 
\begin{lemma}\label{Claim_about impact of zero diagonal entries for a positive semidefinite matrix}
Let $A$ be a $k\times k$ positive semidefinite matrix with real entries.  Let $(i,i)$-th entry of $A$ be zero, for some $1\leq i \leq k$.  Then the $i$-th column and the $i$-th row of $A$ is zero. 
\end{lemma}
\begin{proof}
Follows from Lemma \ref{Generalized Sylvester's criterion for psd matrix}.  Indeed, suppose the contrary.  Then there will be some $1\leq j \leq k$ such that $(i,j)$-th entry, and hence $(j,i)$-th entry of $A$ is nonzero and both are equal.  Then, the determinant of the $2\times 2$ principal submatrix of $A$ formed by $(i,i)$-th entry, $(i,j)$-th entry, $(j,i)$-th entry and $(j,j)$-th entry is negative, which contradicts Lemma \ref{Generalized Sylvester's criterion for psd matrix}.   
\end{proof}
We end this subsection by recalling a couple more well-known characterizations of a positive semidefinite matrix in terms of a given positive semidefinite principal minor and its Schur complement.  
\begin{lemma}\label{Lemma_psdness in terms of leading principal minor and Schur complement}
     Let $M$ be a symmetric matrix with real entries of the following form 
\begin{equation*}
M=
\left(\begin{array}{@{}c|c@{}}
  \begin{matrix}
  A
  \end{matrix}
  & B \\
\hline 
  B^{t} & 
  \begin{matrix}
  C
  \end{matrix}
\end{array}\right)\;,    
\end{equation*}
where $A$ and $C$ are square matrices. Then $M$ is positive semidefinite (respectively positive definite) if and only if $A$ is positive definite and $M/A(:=C-B^tA^{-1}B)$ is positive semidefinite (respectively positive definite).
\end{lemma}
\begin{proof}
See \cite[Theorem 1.12, p.~34]{Zh}.    
\end{proof}
In fact, this lemma has a generalisation. By $A^-$ we denote \textit{a generalised inverse} of a given square matrix $A$, i.e, $AA^-A=A$.  Given a matrix $B$, not necessarily square, by $\mathcal{C}(B)$ we denote its column space.
\begin{lemma}\label{Lemma_psdness in terms of leading principal minor, column space and Schur complement}
 Let $M$ be a symmetric matrix with real entries of the following form 
\begin{equation*}
M=
\left(\begin{array}{@{}c|c@{}}
  \begin{matrix}
  A
  \end{matrix}
  & B \\
\hline 
  B^{t} & 
  \begin{matrix}
  C
  \end{matrix}
\end{array}\right)\;,    
\end{equation*}
where $A$ and $C$ are square matrices. Then $M$ is positive semidefinite if and only if $A$ is positive semidefinite, $\mathcal{C}(B)\subseteq \mathcal{C}(A)$, and $M/A(:=C-B^tA^-B)$ is positive semidefinite.
\end{lemma}
\begin{proof}
    See \cite[Theorem 1.20, p.~44]{Zh}.
\end{proof}

\subsection{Gram Matrix of a symmetric polynomial and a SOHS}
\label{subsec: Gram Matrix of a symmetric polynomial and a SOHS}
Though it is a well known fact that any symmetric noncommutative polynomial has a symmetric matrix associated to it, known as Gram matrix, we still include it in a bit detail as we need this later.
\begin{theorem}\label{Lemma_general form and gram matrix of a symmetric polynomial}
Any $f\in \Sym\;\mathbb{R}\langle \underline{X} \rangle$ of degree $2d$ in $n$ variables has a representation $\mathbf{R}_f$ given as follows:
\begin{equation}\label{equation_general form of a symmetric polynomial}
\mathbf{R}_f\; :\;f=\sum_{i=1}^{s(d,n)}a_i\zeta_i^{\ast}\zeta_i+\sum_{1\leq i < j\leq s(d,n)} b_{ij}(\zeta_i^{\ast}\zeta_j+\zeta_j^{\ast}\zeta_i),
\end{equation}
where 
\begin{equation}\label{conditions on the coefficients of general form of a representation of a symmetric polynomial}
\begin{split}
&s(d,n)=\sum_{k=0}^dn^k\;,\\
& a_i,b_{ij}\in \mathbb{R}\;,\zeta_i\in \mathbb{R}\langle \underline{X}\rangle,\; \deg(\zeta_i)\leq d,\;\zeta_i\neq \zeta_j \text{\;for\;all\;}1\leq i <j \leq s(d,n)\;.
\end{split}
\end{equation}
Hence, there exists a $s(d,n)\times s(d,n)$ symmetric matrix $G_f$ such that
\begin{equation}\label{Gram matrix for a symmetric polynomial}
f=W_{s(d,n)}^{\ast}\;G_f\;W_{s(d,n)},
\end{equation}
where $W_{s(d,n)}$ is the column vector of all monomials of degree at most $d$ taken in lexicographic order.
\end{theorem}
\begin{proof}
Firstly, it is easy to note that any degree $2k$ monomial can be written as a product of two degree $k$ monomials. Now, let $a\eta$ be a term of $f$.  As $f$ is chosen from $\Sym\;\mathbb{R}\langle \underline{X}\rangle$, then either $\eta^{\ast}=\eta$ or $a\eta^{\ast}$ must be another term of $f$.  Hence, the first part of the assertion follows.

Let $1=\zeta_1< \zeta_2< \cdots <\zeta_{s(d,n)}=X_n^d$ be all such monomials in ascending order, where we denote the lexicographic order in the set of all monomials by $<$.  Then we have,
\begin{equation*}\label{general form of a monomial vector of a symmetric polynomial}
W_{s(d,n)}=\left(\begin{array}{@{}c}
  \begin{matrix}
  \zeta_1 & \zeta_2 & \cdots & \zeta_{s(d,n)}
  \end{matrix}\end{array}\right)^{t}.
\end{equation*}
Choosing  
\begin{equation}\label{general form of a Gram matrix of a symmetric polynomial}
G_f=\begin{pmatrix}
a_1 & b_{12} & b_{13} & \cdots & b_{1s(d,n)}\\
  b_{12} & a_2 & b_{23} & \cdots & b_{2s(d,n)}\\
  b_{13} & b_{23} & a_3 & \cdots & b_{3s(d,n)}\\
  \vdots & \vdots & \vdots & \ddots & \vdots\\
  b_{1s(d,n)} & b_{2s(d,n)} & b_{3s(d,n)} & \cdots & a_{s(d,n)}
\end{pmatrix}\;,
\end{equation}
we have the second part of the assertion.
\end{proof}

\begin{definition}
Let $f$ be any symmetric noncommutative polynomial.  The matrix $G_f$, as in \eqref{general form of a Gram matrix of a symmetric polynomial}, is called \textit{a Gram matrix of} $f$.  It is also called the Gram matrix associated with the representation $\mathbf{R}_f$, as in \eqref{equation_general form of a symmetric polynomial}, along with conditions as in \eqref{conditions on the coefficients of general form of a representation of a symmetric polynomial}. 
\end{definition}
\begin{remark}\label{equation_properties of representation}
\begin{enumerate}
\item It can be noted that given a degree $2d$ symmetric noncommutative polynomial, it doesn't necessarily have unique representation.  For example, consider $f=5X_1^6+8X_2^4$.  Then $f$ has representations $5(X_1^3)^{\ast}X_1^3+4(X_2)^{\ast}X_2^3+4(X_2^3)^{\ast}X_2$ and $5(X_1^3)^{\ast}X_1^3+4(X_2^2)^{\ast}X_2^2+4(X_2^2)^{\ast}X_2^2$.
\item However, the moment we fix a representation of a degree $2d$ symmetric noncommutative polynomial, associated Gram matrix and monomial vector get fixed. 
\end{enumerate}
\end{remark}

It is easy to note that any SOHS is definitely a symmetric noncommutative polynomial.  We want to look at the converse.  The following theorem gives a criteria for a symmetric polynomial to be a SOHS in terms of its Gram matrix.
\begin{theorem}\label{Theorem_SOHS iff psd Gram matrix}
A symmetric polynomial $f$ of degree $2d$ in noncommuting variables is a SOHS if and only if there exists a positive semidefinite matrix $G_f$ such that 
\begin{equation}\label{SOHS iff psd Gram matrix_equation}
f=W_{s(d,n)}^{\ast}\;G_f\;W_{s(d,n)},
\end{equation}
where $W_{s(d,n)}$ is the column vector consisting of all words in $\langle \underline X \rangle$ of degree $\leq d$ taken in lexicographic order.
\end{theorem}
\begin{proof}
See \cite[\S 2, p. 105-110]{CLR} or \cite[Proposition 1.16, p.~6]{BKP}.
\end{proof}
\begin{remark}\label{Remark_doesn't guarantee the non-negativity of the matrix}
To express a symmetric polynomial of degree $2d$ in the form as in \eqref{Gram matrix for a symmetric polynomial}, it is not always necessary to take all monomials of degree at most $d$ in $W_{s(d,n)}$.  For example, consider the polynomial $f= 1+2X_1+X_1^2+X_1X_2^2+2X_2^2+X_2^2X_1+X_2X_1^2X_2+X_2^4$.  Then for 
\begin{equation*}
W_5=
\begin{pmatrix}
1 & X_1 &  X_1X_2 & X_2 & X_2^2
\end{pmatrix}^{t}\;,
\end{equation*} 
and 
\begin{equation*}
G=\begin{pmatrix}
1 & 1 & 0 & 0 & 0\\
1 & 1 & 0 & 0 & 1\\
0 & 0 & 1 & 0 & 0\\
0 & 0 & 0 & 2 & 0\\
0 & 1 & 0 & 0 & 1
\end{pmatrix}\;,
\end{equation*}
we have $f=W_5^{\ast}\;G\;W_5$.  But the following Gram matrix of $f$ \begin{equation*}
G_f=\begin{pmatrix}
1 & 1 & 0 & 0 & 0 & 0 & 0\\
1 & 1 & 0 & 0 & 0 & 0 & 1\\
0 & 0 & 0 & 0 & 0 & 0 & 0\\
0 & 0 & 0 & 1 & 0 & 0 & 0\\
0 & 0 & 0 & 0 & 2 & 0 & 0\\
0 & 0 & 0 & 0 & 0 & 0 & 0\\
0 & 1 & 0 & 0 & 0 & 0 & 1
\end{pmatrix}
\end{equation*}
is of much bigger size.
\end{remark}
As the choice of the monomial vector $W_{s(d,n)}$ is very strict in Theorem \ref{Lemma_general form and gram matrix of a symmetric polynomial} and Theorem \ref{Theorem_SOHS iff psd Gram matrix}, there are several techniques to throw out the unnecessary monomials and making the matrix of smaller size.  For example Newton chip method, augmented Newton chip method etc (cf. \cite[Section 2.3 \& 2.4]{BKP}) has been used to reduce the size of the column vector $W_{s(d,n)}$, as in \eqref{SOHS iff psd Gram matrix_equation}, and hence the size of matrix $G_f$.  Moreover, while putting the monomials in $W_{s(d,n)}$, usually lexicographic order is followed.  But as long as enough monomials is being collected in the monomial vector to express a given SOHS $f$,  Lemma \ref{psdness is prserved by same set of row and column interchange} says that the order doesn't matter.  We demonstrate this through the following example.
\begin{example}\label{example_size of the monomials doesn't matter}
Consider $f=X_1^{2}+4X_1X_2+4X_2X_1+16X_2^2$.  Then, for $(W^{\prime}_2)^{\ast}=\begin{pmatrix}
X_1 & X_2
\end{pmatrix}$, $(W^{\prime \prime}_2)^{\ast}=\begin{pmatrix}
X_2 & X_1
\end{pmatrix}$, $G^{\prime}=\begin{pmatrix}
1 & 4\\
4 & 16
\end{pmatrix}
$ and $G^{\prime \prime}=\begin{pmatrix}
16 & 4\\
4 & 1
\end{pmatrix}
$, we have  $ (W^{\prime}_2)^{\ast}\;G^{\prime}\;W_2^{\prime}=f=(W^{\prime \prime}_2)^{\ast}\;G^{\prime \prime}\;W_2^{\prime \prime}$.   
\end{example}
It is evident from Remark \ref{Remark_doesn't guarantee the non-negativity of the matrix} and Example \ref{example_size of the monomials doesn't matter} that to be able to reduce the size of the monomial vector (and hence that of a Gram matrix) whenever possible and write the monomial entries in that monomial vector in any arbitrary order, while expressing a symmetric or SOHS polynomial in the required form, puts us in a resilient situation.  Therefore, we formalize all of this in the upcoming section after putting things in a particular perspective.

\section{Extending a noncommutative polynomial to a SOHS}
\label{sec: how to add new monomials to get a SOHS}
    
In this section, whenever we consider a noncommutative polynomial $f\in \mathbb{R}\langle \underline{X} \rangle$, where $\langle \underline{X} \rangle=\langle X_1,X_2,\dots,X_n \rangle$, we mean that all the variables $X_i$, $1\leq i \leq n$, are essential to express $f$.  Moreover, we want to find a SOHS extension $\widetilde{f}$ of $f$ in $\mathbb{R}\langle \underline{X}\rangle$ only.  Specifically, we now prepare the ground for defining Gram-like matrices, which will help us answer the Question \ref{Introduction_main question_possible modification} that captures the main theme of this article.


\subsection{Gram-like matrices for symmetric and SOHS noncommutative polynomials}
\label{subsec:  matrix}
Recall that Theorem \ref{Theorem_SOHS iff psd Gram matrix} provides a way of getting a SOHS polynomial of degree $2d$ in $n$ variables whenever we have a $s(d,n)\times s(d,n)$ positive semidefinite matrix.  We observe that a positive semidefinite matrix of arbitrary size also produces a SOHS of correct degree.   
\begin{proposition}\label{proposition_order of monomials and size of monomial vector do not matter}
Let $G$ be a $k\times k$ positive semidefinite matrix and $W_k$ be a column vector of size $k$ with $k$ many distinct monomial entries (not necessarily in lexicographical order).  Then 
\begin{enumerate}
\item The polynomial $W_k^{\ast}\;G\;W_k$ is a SOHS polynomial. 
\item Moreover, if no row (and therefore column) of $G$ is zero, then degree of the polynomial $W_k^{\ast}\;G\;W_k$ is twice the degree of any of the highest degree monomials in $W_k$.
\end{enumerate}   
\end{proposition}
\begin{proof}
\begin{enumerate}
\item  Let $d$ be the highest degree among the degrees of the constituent monomials of $W_k$.  Let $W_{s(d,n)}$ be the monomial vector formed by taking the remaining $s(d,n)-k$ many monomials of degree at most $d$ along with the monomials that are already present in $W_k$ and then rearranging them in lexicographic order.  This rearrangement is nothing but a permutation between $\{1,2,\cdots,s(d,n)\}$ to itself and we denote it by $\sigma \in S_{s(d,n)}$, the symmetric group on $s(d,n)$ many symbols.  Let $G^{\prime}$, a $s(d,n)\times s(d,n)$ matrix, be defined as follows:
\begin{equation*}
G^{\prime}:=
\left(\begin{array}{@{}c|c@{}}
  \begin{matrix}
  G
  \end{matrix}
  & \bigzero \\
\hline
  \bigzero &
  \bigzero
\end{array}\right)\;.
\end{equation*}
Let ${G^{\prime}}^{\sigma}$, another $s(d,n)\times s(d,n)$ matrix, be obtained from $G^{\prime}$ by permuting its rows and columns by $\sigma$.  Then ${G^{\prime}}^{\sigma}$ is positive semidefinite as $G$ is so and by Lemma \ref{Lemma_block psd(pd)implies psd(pd)} \& Lemma \ref{psdness is prserved by same set of row and column interchange}.  Moreover, we have the following equality:
\begin{equation*}
W_k^{\ast}\;G\;W_k=W_{s(d,n)}^{\ast}\;{G^{\prime}}^{\sigma}\;W_{s(d,n)}\;.
\end{equation*}
Therefore, the assertion follows from Theorem \ref{Theorem_SOHS iff psd Gram matrix}.

\item Let $\zeta_{i_0}$ be one of the monomials in $W_k$ of maximum degree.  Then, by hypothesis, the $i_0$-th column and $i_0$-th row is nonzero.  Therefore, by Lemma \ref{Claim_about impact of zero diagonal entries for a positive semidefinite matrix}, the $(i_0,i_0)$-th entry of $G$ is positive, say $a_{i_0}$.  Then, $f$ contains the term $a_{i_0}\zeta_{i_0}^{\ast}\zeta_{i_0}$ of degree $2\deg(\zeta_{i_0})$.  So, $\deg(f)\geq 2\deg(\zeta_{i_0})$.  The inequality $\deg(f)\leq 2\deg(\zeta_{i_0})$ is obvious as $W_k$ contains no monomial of degree bigger than $\deg(\zeta_{i_0})$.  Hence, the assertion follows.
\end{enumerate}
\end{proof}
\begin{corollary}\label{corollary_writing a symmetric polynomial without any constraint on the length of the monomial vector}
Let $G$ be any $k\times k$ symmetric matrix and $W_k$ be any column vector of size $k$ with $k$ many distinct monomial entries (not necessarily in lexicographical order).  Then the polynomial $W_k^{\ast}\;G\;W_k$ is a symmetric polynomial.
\end{corollary}
\begin{proof}
Proof goes along the same line of the proof of the first part of the assertion of Proposition \ref{proposition_order of monomials and size of monomial vector do not matter}.
\end{proof}
\begin{remark}\label{remark_immediately related to proposition_order of monomials and size of monomial vector do not matter}
\begin{enumerate}
\item For a positive definite matrix $G$, every diagonal entry, being a principal minor of $G$, is positive.  Therefore, no column or row of $G$ is zero by Lemma \ref{Claim_about impact of zero diagonal entries for a positive semidefinite matrix}.  So, the degree of the SOHS $W_k^{\ast}\;G\;W_k$ is twice the degree of any of the highest degree monomials in $W_k$.
\item Assertion $(2)$ of Proposition \ref{proposition_order of monomials and size of monomial vector do not matter} doesn't hold for a symmetric matrix in general.  Let us take 
\begin{equation*}
W_4:=\begin{pmatrix}
1 & X_1 & X_2 & X_2^2X_3^2
\end{pmatrix}^{t}\;,
\end{equation*}
and 
\begin{equation*}
G:=\begin{pmatrix}
5 & 0 & 0 & 0\\
0 & 1 & 0 & 3\\
0 & 0 & 1 & 0\\
0 & 3 & 0 & 0
\end{pmatrix}\;.
\end{equation*} 
Then the polynomial $f=W_4^{\ast}\; G\; W_4=X_1^2+X_2^2+3X_1X_2^2X_3^2+3X_3^2X_2^2X_1+5$ is a symmetric polynomial of degree $5$ and the highest degree monomial $X_2^2X_3^2$ has degree $4$.
\end{enumerate}
\end{remark}
Let $W_k$ be a monomial vector with degree of its monomial entries being at most $d$ and $G$ be a $k\times k$ symmetric matrix.  Then, as observed in $(2)$ of Remark \ref{remark_immediately related to proposition_order of monomials and size of monomial vector do not matter}, the symmetric polynomial $W_k^{\ast}\;G\;W_k$ might have degree much lesser than $2d$, unlike the positive semidefinite case.  But it is not so under some mild conditions on the symmetric matrix involved.
\begin{proposition}\label{proposition_for positive diagonal symmetric matrix monomials of at most half degree is enough}
Let $G$ be a symmetric matrix with positive diagonal entries and $W_k$ be a column vector of distinct monomials from $\mathbb{R}\langle\underline{X}\rangle$.  Then the polynomial $f:=W_k^{\ast}\;G\;W_k$ is a symmetric polynomial of degree $2d$, $d$ being the degree of any of the highest degree monomial entries of $W_k$.
\end{proposition}
\begin{proof}
Firstly, such an $f$ is a symmetric polynomial by Corollary \ref{corollary_writing a symmetric polynomial without any constraint on the length of the monomial vector}.  Now let $\zeta_{i}$ be a monomial entry of $W_k$ having highest degree, say $d$, among all the entries of $W_k$ for some $1\leq i \leq k$.  Let $a_i(>0)$ be the $(i,i)$-th entry of $G$.  Then $a_i\zeta^{\ast}\zeta$ is a term of $f$ and hence $\deg(f)\geq 2d$.  The other inequality is obvious as no monomial entry of $W_k$ has degree more than $d$.
\end{proof}
We now record the upshot of the findings of this subsection through a couple of theorems. 
\begin{theorem}\label{Theorem_general form of a SOHS polynomial}
\begin{enumerate}
\item Let $G$ be a $k\times k$ positive semidefinite matrix with no row (or column) being zero and $W_k$ be a column vector of size $k$ with $k$ many distinct monomial entries of degree at most $d$.  Then the polynomial $f=W_k^{\ast}\;G\;W_k$ is a SOHS polynomial of degree $2d$ with a representation $\mathbf{R}_f$ given as follows :
\begin{equation}\label{equation_ most general form of a representation of a SOHS polynomial}
\mathbf{R}_f\; :\; f=\sum_{i=1}^{k}a_i\zeta_i^{\ast}\zeta_i+\sum_{1\leq i < j\leq k} b_{ij}(\zeta_i^{\ast}\zeta_j+\zeta_j^{\ast}\zeta_i),
\end{equation}
where
\begin{subequations}\label{conditions on the coefficients of general form of a representation of a SOHS polynomial}
\begin{equation}\label{subeqn : a}
a_i,b_{ij}\in \mathbb{R}, \zeta_i\in \mathbb{R}\langle \underline{X}\rangle,\; \deg(\zeta_i)\leq d \text{\;for\;all\;}1\leq i<j \leq k\;,
\end{equation}
\begin{equation}\label{subeqn : a.5}
    \zeta_i\neq \zeta_j \text{\;for\;all\;}1\leq i\neq j \leq k\;,
\end{equation}
\begin{equation}\label{subeqn : b}
a_i>0 \text{\;for\;all\;}1\leq i \leq k, \text{\;and\;},
\end{equation}
\begin{equation}\label{subeqn : c}
\begin{split}
& \text{for\;any\;} K=\{i_0,i_1,\ldots, i_{\mid K \mid}\}\subset \{1,2,\ldots, k\} \text{\;with\;} i_0<i_1<\cdots <i_{\mid K \mid}\\
&  \text{the\;matrix\;}
\begin{pmatrix}
 a_{i_1} & b_{i_1i_2} & \cdots & b_{i_1i_{\mid K \mid}}\\
 b_{i_1i_2} & a_{i_2} & \cdots  & b_{i_2i_{\mid K \mid}}\\
  \vdots & \vdots & \ddots & \vdots\\
  b_{i_1i_{\mid K \mid}} & b_{i_2i_{\mid K \mid}}  & \cdots & a_{i_{\mid K \mid}}
\end{pmatrix}
\text{is\; positive\;semidefinite\;.}
\end{split}
\end{equation}
\end{subequations}
\item Conversely, let $f$ be a SOHS polynomial of degree $2d$ with a representation $\mathbf{R}_f$, as in \eqref{equation_ most general form of a representation of a SOHS polynomial}, satisfying the conditions as given in \eqref{conditions on the coefficients of general form of a representation of a SOHS polynomial}.  Then there exists a $k\times k$ positive semidefinite matrix $G$ with no row (or column) being zero and a column vector $W_k$ of size $k$ with distinct monomial entries of degree at most $d$ such that  $f=W_k^{\ast}\;G\;W_k$. 
\end{enumerate}
\end{theorem}
\begin{proof}
The condition as in \eqref{subeqn : b} ensures that no column or row are full zero, that is, 
\begin{equation*}
\sum_{i\neq j, i=1}^k b_{ij}^2+a_j^2\neq 0 \text{\;for\;all\;} j\;,\text{and\;},\&\;\sum_{j\neq i, j=1}^k b_{ij}^2+a_i^2\neq 0 \text{\;for\;all\;}i\;,
\end{equation*}
by Proposition \ref{proposition_for positive diagonal symmetric matrix monomials of at most half degree is enough} and Lemma \ref{Claim_about impact of zero diagonal entries for a positive semidefinite matrix}.  Moreover, the condition as in \eqref{subeqn : c} ensures that the matrix whose $(i,j)$-th entries are $b_{ij}$ for all $i\neq j$, and $(i,i)$-th entries are $a_i$ for all $i$, is positive semidefinite by Lemma \ref{Generalized Sylvester's criterion for psd matrix}.  Now it follows from Proposition \ref{proposition_order of monomials and size of monomial vector do not matter}.
\end{proof}
\begin{theorem}\label{Theorem_general form of a symmetric polynomial of degree 2d where the co-efficients of the square terms are positive}
\begin{enumerate}
\item Let $G$ be a $k\times k$ symmetric matrix with positive diagonal entries and $W_k$ be a column vector of size $k$ with distinct monomial entries of degree at most $d$.  Then the polynomial $f=W_k^{\ast}\;G\;W_k$ is a symmetric polynomial of degree $2d$ with a representation $\mathbf{R}_f$ given as follows :
\begin{equation}\label{equation_ most general form of a representation of a symmetric polynomial with positive co-efficients of square terms}
\mathbf{R}_f\; :\; f=\sum_{i=1}^{k}a_i\zeta_i^{\ast}\zeta_i+\sum_{1\leq i < j\leq k} b_{ij}(\zeta_i^{\ast}\zeta_j+\zeta_j^{\ast}\zeta_i),
\end{equation}
where
\begin{subequations}\label{conditions on the coefficients of general form of a representation of a symmetric polynomial with positive co-efficients of square terms}
\begin{equation}\label{subeqn : a'}
a_i,b_{ij}\in \mathbb{R}, \zeta_i\in \mathbb{R}\langle \underline{X}\rangle,\;\zeta_i\neq \zeta_j,\; \deg(\zeta_i)\leq d \text{\;for\;all\;}1\leq i<j \leq k\;,
\end{equation}
\begin{equation}\label{subeqn : b'}
a_i>0 \text{\;for\;all\;}1\leq i \leq k\;,
\end{equation}
\end{subequations}
\item Conversely, let $f$ be a symmetric polynomial of degree $2d$ with a representation $\mathbf{R}_f$, as in \eqref{equation_ most general form of a representation of a symmetric polynomial with positive co-efficients of square terms}, satisfying the conditions as given in \eqref{conditions on the coefficients of general form of a representation of a symmetric polynomial with positive co-efficients of square terms}.  Then there exists a $k\times k$ symmetric matrix $G$ with positive diagonal entries and a column vector $W_k$ of size $k$ with distinct monomial entries of degree at most $d$ such that  $f=W_k^{\ast}\;G\;W_k$. 
\end{enumerate}
\end{theorem}
\begin{proof}
The condition as in \eqref{subeqn : b'} ensures that no column or row are full zero.  The assertion now follows from Proposition \ref{proposition_for positive diagonal symmetric matrix monomials of at most half degree is enough}.
\end{proof}
\begin{definition}\label{Definition_ matrix_positive diagonal}
Let $f\in \mathbb{R}\langle \underline{X}\rangle$ be a symmetric polynomial of degree $2d$ in $n$ variables and $k$ be any positive integer satisfying $k\leq s(d,n)$.  If there exist a $k\times k$ symmetric matrix $G$ with positive diagonal entries such that $f$ can be written in the form $W_k^{\ast}\;G\;W_k$ for some column vector $W_k$ of monomials of degree at most $d$, then $G$ is called \textit{a Gram-like matrix of} $f$ (or \textit{the Gram-like matrix associated with $\mathbf{R}_f$, as in
\eqref{equation_ most general form of a representation of a symmetric polynomial with positive co-efficients of square terms}, along with conditions as in \eqref{conditions on the coefficients of general form of a representation of a symmetric polynomial with positive co-efficients of square terms}}).  
\end{definition}
\begin{definition}\label{Definition_Gram like matrix_psd}
Let $f\in \mathbb{R}\langle \underline{X}\rangle$ be a SOHS polynomial of degree $2d$ in $n$ variables and $k$ be any positive integer satisfying $k\leq s(d,n)$.  If there exists a $k\times k$ positive semidefinite matrix $G$ such that $f$ can be written in the form $W_k^{\ast}\;G\;W_k$ for some column vector $W_k$ of monomials of degree at most $d$, then $G$ is called \textit{a positive semidefinite Gram-like matrix of} $f$ (or \textit{the positive semidefinite Gram-like matrix associated with $\mathbf{R}_f$, as in
\eqref{equation_ most general form of a representation of a SOHS polynomial}, along with conditions as in \eqref{conditions on the coefficients of general form of a representation of a SOHS polynomial}}).  
\end{definition}
\begin{remark}\label{remark_comparing Gram matrix and Gram like matrix}
\begin{enumerate}
\item The difference between a Gram matrix and a Gram-like matrix of a polynomial can be seen through the following example.  A positive semidefinite Gram-like matrix of the polynomial $f=X_1^{2}+4X_1X_2+4X_2X_1+16X_2^2$, as in Example \ref{example_size of the monomials doesn't matter}, is $G^{\prime \prime}=\begin{pmatrix}
16 & 4\\
4 & 1
\end{pmatrix}
$, whereas each of its positive semidefinite  Gram matrix is of size $7$. 

\item It can be noted from the proof of Proposition \ref{proposition_order of monomials and size of monomial vector do not matter} that a positive semidefinite Gram-like matrix associated to a SOHS polynomial is equally important to its positive semidefinite Gram matrix as it encodes all the information of its Gram matrix.  Moreover, it is easier to handle a Gram-like matrix as it usually has much lesser size than corresponding Gram matrix. 

\item The motivation for working with a Gram-like matrix of a noncommutative polynomial $f$ instead of its Gram matrix is to avoid unnecessary monomials in the corresponding monomial vector, which is same as ignoring any full zero row (or column) from a Gram matrix.  The transition from a Gram matrix to a Gram-like matrix can be interpreted through the lens of graph theory as well.  To be specific, the sparsity pattern (cf. \cite[\S 8.1, p. 329-330]{VaAn} for more details) of a Gram-like matrix is always connected, whereas it is not so for a Gram matrix in general.   
\end{enumerate}
\end{remark}

\subsection{Gram-like matrix preserving SOHS extensions}
\label{subsec: Block Matrix Technique}
In this subsection, we answer Question \ref{Introduction_main question_possible modification} completely by exploiting the flexibility obtained through Theorem \ref{Theorem_general form of a SOHS polynomial} and Theorem \ref{Theorem_general form of a symmetric polynomial of degree 2d where the co-efficients of the square terms are positive}.  To do so, let us formally define the meaning of SOHS extension of a noncommutative polynomial.

\begin{definition}
Let $f\in \mathbb{R}\langle \underline{X}\rangle$ be a noncommutative polynomial of the form $f=\sum_{j=1}^m a_j \delta_j$, where $\delta_j$ are monomials in $\mathbb{R}\langle\underline{X}\rangle$, $a_j\in \mathbb{R}-\{0\}$ for all $j$.  A SOHS polynomial $\widetilde{f}\in \mathbb{R}\langle \underline{X}\rangle$ is called  \textit{a nontrivial SOHS extension of} $f$ if $\widetilde{f}-f=\sum_{i=1}^{n}b_i\alpha_i$, where $\alpha_i$ are monomials in $\mathbb{R}\langle\underline{X}\rangle$, $\alpha_i\neq \delta_j$, $b_i\in \mathbb{R}-\{0\}$ for all $i,j$.    
\end{definition}

From now on, by a SOHS extension we mean a nontrivial SOHS extension only, as defined above.  We now define a special type of SOHS extension, namely, a Gram-like matrix preserving extension.
\begin{definition}\label{definition_Gram-like matrix preserving extension}
    A SOHS extension $\widetilde{f}\in \mathbb{R}\langle \underline{X}\rangle$ of a noncommutative polynomial $f\in \mathbb{R}\langle \underline{X}\rangle$ is said to be \textit{a Gram-like matrix preserving extension of} $f$ \textit{w.r.t} $h$ if 
    \begin{itemize}
        \item  there exists a $l \times l$ positive semidefinite Gram-like matrix $G_{\widetilde{f}}$ of $\widetilde{f}$ that keeps a $k\times k$ positive semidefinite Gram-like matrix $G_h$ of a SOHS part $h$ of $f$ as a principal submatrix, for some positive integers $l>k$, and 
        \item there exists a monomial vector $W_l$ associated to $G_{\widetilde{f}}$ and a monomial vector $W_k$ associated to $G_h$ such that $W_l$ contains all the entries of $W_k$.
    \end{itemize} 
\end{definition}
\begin{remark}\label{remark_Gram-like matrix preserving extension}
    From now on, we simply say that $G_{\widetilde{f}}$ is a positive semidefinite Gram-like matrix of a Gram-like matrix preserving extension $\widetilde{f}$ of $f$ to mean that such
    $G_{\widetilde{f}}$ satisfies both the conditions as in Definition \ref{definition_Gram-like matrix preserving extension}.
\end{remark}
\begin{example}\label{Examples_polynomials that do not admit G-L MPE}
\begin{enumerate}
 \item There are noncommutative polynomials with a SOHS part $h$ that do not admit any Gram-like matrix preserving extension w.r.t $h$.  For example, consider the polynomial $f=X^2+X+5$ with $h=X^2+5$.  It is easy to check that only possible positive semidefinite Gram-like matrices of $h$ are $\begin{pmatrix}
1 & 0\\
0 & \sqrt{5}
\end{pmatrix}$ and $\begin{pmatrix}
\sqrt{5} & 0\\
0 & 1
\end{pmatrix}$.  The corresponding monomial vectors are $\begin{pmatrix}
X & 1
\end{pmatrix}^{t}$ and $\begin{pmatrix} 1 & X\end{pmatrix}^{t}$ respectively.  Then $f$ doesn't admit a Gram-like matrix preserving extension w.r.t $h$.  This is because only possible ways to express $X$ as a product of two monomials are $1.X$ and $X.1$.  That is, in positive semidefinite Gram-like matrix off diagonal terms are getting affected.  Using exactly similar argument, it is easy to observe that the polynomials $X_1^2+X_2^2+5+X_1X_2$, $X_1^2+X_2^2+7+X_1X_2X_1$ do not admit admit any Gram-like matrix preserving extension w.r.t $h:=X_1^2+X_2^2+5$.
\item By definition, it is obvious that any Gram-like matrix preserving extension w.r.t some SOHS part of a given noncommutative polynomial is a SOHS extension.  However, the converse is not true.  Consider the polynomial $f=X^2+X+5$ with $h=X^2+5$.  Then $\widetilde{f}=X^2+X+5+X^4=f+X^4$ is a SOHS extension of $f$.  But, as explained above, $f$ doesn't admit a Gram-like matrix preserving extension w.r.t $h$. 
\end{enumerate}
\end{example}


The following lemma talks about the general form of a positive semidefinite Gram-matrix of a Gram-like matrix preserving extension. 
\begin{lemma}\label{Lemma_form of Gram-like matrix matrix of Gram-like matrix preserving extension}
Let $f\in \mathbb{R}\langle\underline{X}\rangle$ be a noncommutative polynomial and $h$ a SOHS part of $f$.  Let $\widetilde{f}$ be a Gram-like matrix preserving extension of $f$ w.r.t $h$ and $G_h$ and $G_{\widetilde{f}}$ be as in Definition \ref{definition_Gram-like matrix preserving extension}.  Then, up to permutation, $G_{\widetilde{f}}$ must be of the form 
\begin{equation*}
G_{\widetilde{f}}=
\left(\begin{array}{@{}c|c@{}}
  \begin{matrix}
  G_h
  \end{matrix}
  & A \\
\hline 
  A^{t} & 
  \begin{matrix}
  B
  \end{matrix}
\end{array}\right)\;,    
\end{equation*}
where $G_{\widetilde{f}}/A(:=A^tB^-A)$ is positive semidefinite and $\mathcal{C}(A)\subseteq \mathcal{C}(G_h)$.
\end{lemma}
\begin{proof}
Follows from Definition \ref{definition_Gram-like matrix preserving extension},  Lemma \ref{psdness is prserved by same set of row and column interchange} and Lemma \ref{Lemma_psdness in terms of leading principal minor, column space and Schur complement}.    
\end{proof}
 We now look for some necessary conditions and sufficient conditions for the existence of a Gram-like matrix preserving extension of a given noncommutative polynomial $f$, in terms of the monomials of $f$.  We set the stage for that.

The following lemma talks about how a given monomial can be decomposed through two smaller sized monomials.
\begin{lemma}\label{Lemma_decomposition of monomials using rc function }
    Any monomial $W\in \langle \underline{X}\rangle$ can be decomposed as a product of two monomials in $\deg(W)+1$ many ways. 
\end{lemma}
\begin{proof}
     Any monomial $ W \in \langle \underline{X}\rangle$ can be decomposed as following:
\begin{equation}\label{splitting a monomial using rc function and conjugate}
W = \big( \rc(W^*, \deg(W)-n) \big)^* \rc(W, n),   \end{equation}
for any integer $n$ satisfying $0\leq n \leq \deg(W)$ and therefore there are $\deg(W)+1$ many such ways. 
\end{proof}
\begin{definition}
    Given a monomial $W\in \langle \underline{X} \rangle$, any decomposition of $W$, as in  \eqref{splitting a monomial using rc function and conjugate}, is called a  \textit{rc deomposition of} $W$ (\textit{Right Chip decomposition of} $W$).
\end{definition}
The following lemma provides a way to produce a positive semidefinite matrix out of a partially specified symmetric matrix of a particular type.  This is very crucial in proving the existence of a Gram-like matrix preserving extension. 
\begin{lemma}\label{Lemma_generating a psd matrix by choosing the diagonals of the lower diagonal block}
 Let \begin{equation*}
G=
\left(\begin{array}{@{}c|c@{}}
  \begin{matrix}
  A
  \end{matrix}
  & B \\
\hline 
  B^{t} & 
  \begin{matrix}
  C
  \end{matrix}
\end{array}\right)\;   
\end{equation*} 
be a partially specified symmetric matrix, where $A$ and $C$ are square matrices of size $m$ and $n$ respectively and $A$ is a positive semidefinite matrix.  Let the diagonal elements of $C$ be the only unspecified entries of $G$.  Then if the column space of $B$ sits inside the column space of $A$, that is, $\mathcal{C}(B)\subseteq \mathcal{C}(A)$, then it is possible to fill the diagonal elements of $C$ with some positive real numbers so that the matrix $G$ becomes positive semidefinite.   
\end{lemma}
\begin{proof}
By Lemma \ref{Lemma_psdness in terms of leading principal minor, column space and Schur complement}, it is enough to choose diagonals $c_{ii}$ of $C$ such that 
\begin{itemize}
    \item $c_{ii}>0$ for all $1\leq i \leq n$,
    \item The matrix $G/A:=C-B^tA^-B$ is positive semidefinite.
\end{itemize}
We choose $c_{ii}>0$ satisfying $c_{ii} \geq \sum_{j \neq i} |(C- B^tA^-B)_{ij}|$, for all $i$.  Clearly, such choices make the matrix $G/A:=C-B^tA^-B$ diagonally dominant, and hence positive semidefinite.  Hence the assertion follows.
\end{proof}
\begin{remark}
    In Lemma \ref{Lemma_generating a psd matrix by choosing the diagonals of the lower diagonal block}, if the matrix $A$ is positive definite, then condition $\mathcal{C}(B)\subseteq \mathcal{C}(B)$ is automatically satisfied.
\end{remark}

We are now ready to provide some necessary and sufficient conditions for the existence of a Gram-like matrix preserving extension.
\begin{theorem}\label{Theorem_conditions satisfied my monomial vectors if if G-l MPE exists}
    Let $f=h+\sum_{j=1}^m d_j\delta_j\in \mathbb{R}\langle \underline{X}\rangle$ be a noncommutative polynomial with a given SOHS part $h$ and the monomials $\delta_j$s are not Hermitian squares.
    \begin{enumerate}
        \item Assume that there exists a Gram-like matrix preserving extension $\widetilde{f}$ of $f$ w.r.t $h$. Then, there exists a representation 
    \begin{equation*}
        \mathbf{R}_h\; :\; h=\sum_{i=1}^{k}a_i\zeta_i^{\ast}\zeta_i+\sum_{1\leq i < j\leq k} b_{ij}(\zeta_i^{\ast}\zeta_j+\zeta_j^{\ast}\zeta_i)\;,
    \end{equation*}
 and, for all  $1\leq j\leq m$, there exist rc decompositions  of $\delta_j=\beta_{n_j}^{\ast}\gamma_{n_j}$, 
 for some $0\leq n_j \leq \deg(\delta_j)$, 
 such that no $\beta_{n_j}$ and $\gamma_{n_j}$  simultaneously  lie in $\{\zeta_1,\ldots,\zeta_k\}$.
 \item Conversely, assume that there exists a representation 
    \begin{equation*}
        \mathbf{R}_h\; :\; h=\sum_{i=1}^{k}a_i\zeta_i^{\ast}\zeta_i+\sum_{1\leq i < j\leq k} b_{ij}(\zeta_i^{\ast}\zeta_j+\zeta_j^{\ast}\zeta_i)\;,
    \end{equation*}
 and, for all  $1\leq j\leq m$, there exist rc decompositions  of $\delta_j=\beta_{n_j}^{\ast}\gamma_{n_j}$, for some $0\leq n_j \leq \deg(\delta_j)$, such that no $\beta_{n_j}$ and $\gamma_{n_j}$  simultaneously  lie in $\{\zeta_1,\ldots,\zeta_k\}$.  Then, there exists a partially specified symmetric matrix $G$ of the form 
 \begin{equation*}
G=
\left(\begin{array}{@{}c|c@{}}
  \begin{matrix}
  G_h
  \end{matrix}
  & A \\
\hline 
  A^{t} & 
  \begin{matrix}
  B
  \end{matrix}
\end{array}\right)\;,   
\end{equation*}
where $G_h$ is the Gram-like matrix of $h$ w.r.t the representation $\mathbf{R}_h$.  Moreover, if $\mathcal{C}(A)\subseteq \mathcal{C}(G_h)$, then there exist infinitely many Gram-like matrix preserving extensions of $f$.
    \end{enumerate}
\end{theorem}
\begin{proof}
\begin{enumerate}
    \item As $\widetilde{f}$ is a Gram-like matrix preserving extension w.r.t $h$, by Lemma \ref{Lemma_form of Gram-like matrix matrix of Gram-like matrix preserving extension} it must be of the form \begin{equation*}
G_{\widetilde{f}}=
\left(\begin{array}{@{}c|c@{}}
  \begin{matrix}
  G_h
  \end{matrix}
  & A \\
\hline 
  A^{t} & 
  \begin{matrix}
  B
  \end{matrix}
\end{array}\right)\;   
\end{equation*}
up to a permutation, where $G_h$ is a positive semidefinite Gram-like matrix of $h$.  Let $G_h$ be the Gram-like matrix associated with 
\begin{equation*}
        \mathbf{R}_h\; :\; h=\sum_{i=1}^{k}a_i\zeta_i^{\ast}\zeta_i+\sum_{1\leq i < j\leq k} b_{ij}(\zeta_i^{\ast}\zeta_j+\zeta_j^{\ast}\zeta_i)\;.
    \end{equation*}
\begin{enumerate}
    \item[Case 1]  All the coefficients $d_j$ of the monomials $\delta_j$ of $f-h$ are distributed in $B$ : \\
    In this case, there exist some $0\leq n_j \leq \deg(\delta_j)$, such that the decompositions $$\delta_j=(\rc(\delta_j^{\ast},\deg(\delta_j)-n_j))^{\ast}\rc(\delta_j,n_j)$$ satisfies \begin{equation*}\label{new monomials are totally different}
    \rc(\delta_j^{\ast},\deg(\delta_j)-n_j)\neq \zeta_i \neq \rc(\delta_j,n_j),
    \end{equation*}
    for all $1\leq i \leq k$ and $1\leq j\leq m$.
    \item[Case 2] Not all the coefficients $d_j$ of the monomials $\delta_j$ of $f-h$ are distributed in $B$ :\\
    In this case, some $d_j$ must be distributed in $A$ or $A^t$ or both.  Therefore, 
    \begin{enumerate}
        \item either there exist decompositions $\delta_j=(\rc(\delta_j^{\ast},\deg(\delta_j)-n_j))^{\ast}\rc(\delta_j,n_j)$, for some $0\leq n_j \leq \deg(\delta_j)$, such that 
    $$\beta_{n_j}(:=\rc(\delta_j^{\ast},\deg(\delta_j)-n_j))= \zeta_i$$
    for some $i \in K_1\subseteq\{1,\ldots, k\}$ and $j \in M_1\subseteq\{1,\ldots, m\}$.  Then, we must have 
                 \begin{equation*}
                   \rc(\delta_j,n_j) \neq \zeta_i\text{\;for\;all\;}i\in K_1\text{\;and\;}j\in M_1\;,
                 \end{equation*}
    as otherwise these $d_j$s would have been distributed inside $G_h$, which is not possible as $\widetilde{f}$ is a Gram-like matrix preserving extension of $f$ w.r.t $h$. 
    \item Exactly similar to (i).
    \end{enumerate}
\end{enumerate}
\item Reversing the implications of the proof the first part of the assertion, it is easy to see that if no $\beta_{n_j}$ and $\gamma_{n_j}$  simultaneously  lie in $\{\zeta_1,\ldots,\zeta_k\}$, then there exists a partially specified symmetric matrix $G$ of the form 
 \begin{equation*}
G=
\left(\begin{array}{@{}c|c@{}}
  \begin{matrix}
  G_h
  \end{matrix}
  & A \\
\hline 
  A^{t} & 
  \begin{matrix}
  B
  \end{matrix}
\end{array}\right)\;,   
\end{equation*}
where $G_h$ is the Gram-like matrix of $h$ w.r.t the representation $\mathbf{R}_h$, such that the coefficient $d_j$s of $\delta_j$s and $\delta_j^{\ast}$s are distributed in $G$ outside the block $G_h$.  Here we are taking $W_l$ as a associated monomial vector to $G$ obtained by taking $\zeta_i$s as the first $k$ constituent monomials and taking $\beta_{n_j}$s and $\gamma_{n_j}$s (without any repetition) as the remaining monomials.  Also $d_j$s can't be distributed in the diagonals of $B$, as $\delta_j$s are not Hermitian squares.

Therefore we now have a partially specified symmetric matrix $G=\left(\begin{array}{@{}c|c@{}}
  \begin{matrix}
  G_h
  \end{matrix}
  & A \\
\hline 
  A^{t} & 
  \begin{matrix}
  B
  \end{matrix}
\end{array}\right)$ such that $G_h$ and $B$ are square matrices and the only unspecified entries of $G$ are the diagonal entries of $B$.  Now, if $\mathcal{C}(A)\subseteq \mathcal{C}(G_h)$, then by Lemma \ref{Lemma_generating a psd matrix by choosing the diagonals of the lower diagonal block}, we can choose diagonal entries of $B$ in such a way that makes $G$ a positive semidefinite matrix.  Therefore, the SOHS $\widetilde{f}:=W_l^{\ast}\;G\;W_l$ is the required Gram-like preserving extension of $f$.

Moreover, it is clear from Lemma \ref{Lemma_generating a psd matrix by choosing the diagonals of the lower diagonal block} that there are infinitely many choices for the diagonal entries of $B$ that makes $G$ positive semidefinite.  Therefore, we obtain infinitely many such Gram-like preserving extension $\widetilde{f}$ of $f$.

 \end{enumerate}
\end{proof}

\begin{remark}\label{Remark_column space condition is nonredundant for semidefinite case}
\begin{enumerate}
\item In Theorem \ref{Theorem_conditions satisfied my monomial vectors if if G-l MPE exists}, the hypothesis that the monomials $\delta_j$s of $f-h$ are non Hermitian is used to prove the converse part only.
\item Whenever the matrix $G_h$ of $h$, as in Theorem \ref{Theorem_conditions satisfied my monomial vectors if if G-l MPE exists}, is positive definite, then the necessary condition itself becomes sufficient as the condition $\mathcal{C}(A)\subseteq \mathcal{C}(G_h)$ is satisfied by default.
\item In the converse part of Theorem \ref{Theorem_conditions satisfied my monomial vectors if if G-l MPE exists}, the condition $\mathcal{C}(A)\subseteq \mathcal{C}(G_h)$ is not redundant when $G_h$ is not positive definite.  To see this, consider the noncommutative polynomial $f=X_1^2+X_2^2+X_1X_2+X_2X_1+2X_1-2Y_2$ and take $h=X_1^2+X_2^2+X_1X_2+X_2X_1$.  It is easy to check that $h$ has a unique positive semidefinite Gram-like matrix $G_h$, namely, $G_h=\begin{pmatrix}
1 & 1\\
1 & 1
\end{pmatrix}$.  The corresponding monomial vectors could be either $\begin{pmatrix}
X_1 & X_2
\end{pmatrix}^{t}$ or $\begin{pmatrix} X_2 & X_1\end{pmatrix}^{t}$.  Also, the only possible rc decompositions of $X_1$ and $X_2$ are $1.X_1$, $X_1.1$, $1.X_2$ and $X_2.1$.    

Assume that there exists a Gram-like matrix preserving extension $\widetilde{f}$ of $f$ w.r.t $h$.  Because of the possible rc decompositions of $X_1$ and $X_2$, as mentioned earlier, $\widetilde{f}$ must have a constant term, say $d$.  Then for any such $\widetilde{f}$, there should be some positive semidefinite Gram-like matrix $G_{\widetilde{f}}$ such that it contains either $\begin{pmatrix}
1 & 1 & 1\\
1 & 1 & -1\\
1 & -1 & d
\end{pmatrix}$ or $\begin{pmatrix}
1 & 1 & -1\\
1 & 1 & 1\\
-1 & 1 & d
\end{pmatrix}$, corresponding to monomial vector $\begin{pmatrix}
X_1 & X_2 & 1
\end{pmatrix}^{t}$ or $\begin{pmatrix}
X_2 & X_1 & 1
\end{pmatrix}^{t}$, as a principal minor.  But any of these principal minors are having determinant $-4 (<0)$. 
 Therefore, by Lemma \ref{Generalized Sylvester's criterion for psd matrix}, such $G_{\widetilde{f}}$, and hence such $\widetilde{f}$, can't exist.
\end{enumerate}
\end{remark}
\begin{example}\label{Example_Block G-l MPE doesn't always exist}
    Consider the polynomial $f=X_1^2+1+X_1X_2$.  We take $h=X_1^2+1$.  Then, $\widetilde{f}=f+X_2X_1+X_2^2$ is a Gram-like matrix preserving extension of $f$ w.r.t $h$ as $G_{\widetilde{f}}$ given as $G_{\widetilde{f}}=\begin{pmatrix}
        1 & 0 & 0 \\
        0 & 1 & 1\\
        0 & 1 & 1
    \end{pmatrix}$ (alongwith $W_3=\begin{pmatrix}
        1 & X_1 & X_2
    \end{pmatrix}^t$ as the associated monomial vector) preserves $G_h=\begin{pmatrix}
        1 & 0\\
        0 & 1
    \end{pmatrix}$ (alongwith $W_2=\begin{pmatrix}
        1 & X_1
    \end{pmatrix}^t$ as the associated monomial vector) as a principal submatrix.  It is easy to see that $G_h$ is unique up to permutation.  Also, possible rc decompositions of $X_1X_2$ are $1. X_1X_2$, $X_1.X_2 $ and $X_1X_2.1$.  Therefore, there does not exist any $G_{\widetilde{f}}$ of the form $G_{\widetilde{f}}=
\left(\begin{array}{@{}c|c@{}}
  \begin{matrix}
  G_{h}
  \end{matrix}
  & \bigzero \\
\hline 
  \bigzero & 
  \begin{matrix}
  G_{\widetilde{f}-h}
  \end{matrix}
\end{array}\right)$.
\end{example}
Now we want to look that under what circumstances, a Gram-like matrix preserving extension $\widetilde{f}$ of $f=h+\sum_{j=1}^m d_j\delta_j$ admits a Gram-like matrix $G_{\widetilde{f}}$ of the form 
$G_{\widetilde{f}}=
\left(\begin{array}{@{}c|c@{}}
  \begin{matrix}
  G_{h}
  \end{matrix}
  & \bigzero \\
\hline 
  \bigzero & 
  \begin{matrix}
  G_{\widetilde{f}-h}
  \end{matrix}
\end{array}\right)$.  Here we omit the additional assumption that the monomials $\delta_j$s are not Hermitian squares.   

\begin{theorem}\label{Theorem_NASC for block matrix form of a Gram-matrix of the extension}
Let $f\in \mathbb{R} \langle \underline X \rangle$ be a noncommutative polynomial.  Let $h$ be a SOHS part of $f$.  Let $\widetilde{f}$ be a SOHS extension of $f$.  Then the following are equivalent :
\begin{enumerate}
    \item $\widetilde{f}$ is a Gram-like matrix preserving extension of $f$ w.r.t $h$ such that there exists a positive semidefinite Gram-like matrix $G_{\widetilde{f}}$ of $\widetilde{f}$ of the form
    \begin{equation}\label{equation_structure of a Gram-like matrix preserving Gram-like matrix}
        G_{\widetilde{f}}=
\left(\begin{array}{@{}c|c@{}}
  \begin{matrix}
  G_{h}
  \end{matrix}
  & \bigzero \\
\hline 
  \bigzero & 
  \begin{matrix}
  G_{\widetilde{f}-h}
  \end{matrix}
\end{array}\right)\;,
    \end{equation}
where $G_{h}$ and $G_{\widetilde{f}-h}$ are $k\times k$ and $l\times l$ positive semidefinite Gram-like matrices of $h$ and $\widetilde{f}-h$ respectively.
\item There exist a representation $\mathbf{R}_h$ of $h$ and and a representation $\mathbf{R}_{\widetilde{f}-h}$ of $\widetilde{f}-h$, given as 
\begin{equation*}
    \begin{split}
      \mathbf{R}_h\; :\; h=&\sum_{i=1}^{k}a_i\zeta_i^{\ast}\zeta_i+\sum_{1\leq i < j\leq k} b_{ij}(\zeta_i^{\ast}\zeta_j+\zeta_j^{\ast}\zeta_i)\;,\\
      \mathbf{R}_{\widetilde{f}-h}\; :\; \widetilde{f}-h=&\sum_{i=1}^{l}c_i\eta_i^{\ast}\eta_i+\sum_{1\leq i < j\leq l} d_{ij}(\eta_i^{\ast}\eta_j+\eta_j^{\ast}\eta_i),
    \end{split}
\end{equation*}
both satisfying conditions as in \eqref{conditions on the coefficients of general form of a representation of a SOHS polynomial}, such that
\begin{equation*}
    \zeta_i\neq \eta_j \;\text{for\;all\;} 1\leq i \leq k, 1\leq j \leq l.
\end{equation*}
\end{enumerate} 
\end{theorem}
\begin{proof}
$(1)\Rightarrow (2)$ :  By Definition \ref{definition_Gram-like matrix preserving extension} and Remark \ref{remark_Gram-like matrix preserving extension}, there exists a monomial vector
\begin{equation*}
    \widetilde{W_{k+l}}:=\begin{pmatrix}
        \zeta_1 & \cdots & \zeta_k & \eta_{1} & \cdots & \eta_{l}
    \end{pmatrix}\;^t 
\end{equation*}
associated to $G_{\widetilde{f}}$ such that $h=W_k^{\ast}\;G_h\;W_k$.  Here $W_k=\begin{pmatrix}
        \zeta_1 & \cdots & \zeta_k
    \end{pmatrix}\;^t$.  Then,
\begin{equation*}
\begin{split}
    \widetilde{f}-h&=\widetilde{W_{k+l}}^{\ast}\;\left(\begin{array}{@{}c|c@{}}
  \begin{matrix}
  G_{h}
  \end{matrix}
  & \bigzero \\
\hline 
  \bigzero & 
  \begin{matrix}
  G_{\widetilde{f}-h}
  \end{matrix}
\end{array}\right)\;\widetilde{W_{k+l}}-W_k^{\ast}\;G_h\;W_k\\
&=W_l^{\ast}\;G_{\widetilde{f}-h}\;W_l\;,
    \end{split}
\end{equation*}    
where $W_l=\begin{pmatrix}
        \eta_{1} & \cdots & \eta_{l}
    \end{pmatrix}\;^t$.  Now, let $\mathbf{R}_h$ and $\mathbf{R}_{\widetilde{f}-h}$ be the associated representations of $G_h$ and $G_{\widetilde{f}-h}$ respectively.  Then, $\mathbf{R}_h$ and $\mathbf{R}_{\widetilde{f}-h}$ must satisfy the conditions as in \eqref{conditions on the coefficients of general form of a representation of a SOHS polynomial}.  Moreover, $\zeta_i\neq \eta_j$ for all $1\leq i \leq k$, $1\leq j \leq l$, as $\widetilde{W_{k+l}}$ is a monomial vector.

$(2)\Rightarrow (1)$ : Let $G_h$ and $G_{\widetilde{f}-h}$ be positive semidefinite Gram-like matrices of $h$ and $\widetilde{f}-h$ associated with representations $\mathbf{R}_h$ and $\mathbf{R}_{\widetilde{f}-h}$ respectively.  As, $\zeta_i \neq \eta_j$ for all $1\leq i \leq k$, $1\leq j \leq l$, we can form the the monomial vector 
\begin{equation*}\label{equation_monomial vector that gives the wanted form of a Gram-like matrix of an extension}
    \widetilde{W_{k+l}}:=\begin{pmatrix}
        \zeta_1 & \cdots & \zeta_k & \eta_1 & \cdots & \eta_l
    \end{pmatrix}\;^t.  
\end{equation*}
Then, denoting $\begin{pmatrix}
        \zeta_1 & \cdots & \zeta_k
    \end{pmatrix}\;^t$ and $\begin{pmatrix}
        \eta_1 & \cdots & \eta_l
    \end{pmatrix}\;^t$ by $W_k$ and $W_l$ respectively, we have:
\begin{equation*}
\begin{split}
    \widetilde{f}&=h+(\widetilde{f}-h)\\
    &=W_k^{\ast}\;G_h\;W_k+W_l^{\ast}\;G_{\widetilde{f}-h}\;W_l\\
    &=\widetilde{W_{k+l}}^{\ast}\;\left(\begin{array}{@{}c|c@{}}
  \begin{matrix}
  G_{h}
  \end{matrix}
  & \bigzero \\
\hline 
  \bigzero & 
  \begin{matrix}
  G_{\widetilde{f}-h}
  \end{matrix}
\end{array}\right)\;\widetilde{W_{k+l}}\;.
    \end{split}
\end{equation*}
\end{proof}
\begin{remark}\label{Remark_producing examples of sohs-sohs=sohs}
Although the sum $h_1+h_2$ of two SOHS polynomials $h_1$ and $h_2$ is always SOHS, the difference $h_1-h_2$ does not enjoy the same.  
An example would be $h_1=1+X_1^2+X_1X_2+X_2X_1+X_2^2$ and $h_2=1+X_1^2$, as $h_1-h_2=X_1X_2+X_2X_1+X_2^2$ is not a SOHS.  Comparing this with Example \ref{Example_Block G-l MPE doesn't always exist}, we conclude that an important application of Theorem \ref{Theorem_NASC for block matrix form of a Gram-matrix of the extension} is to produce nontrivial SOHS polynomials such that their difference is also a SOHS.   \end{remark}

Theorem \ref{Theorem_NASC for block matrix form of a Gram-matrix of the extension} talks about an existential criterion for a Gram-like matrix $G_{\widetilde{f}}$ of the extension $\widetilde{f}$ of $f$ of the form as in \eqref{equation_structure of a Gram-like matrix preserving Gram-like matrix}.  We now present one such method of construction of $G_{\widetilde{f}}$ for a class of noncommutative polynomials.

\begin{theorem}\label{Block PSD technique_with non-SOHS part having multiple terms}
Let $f=h+\sum_{j=1}^ra_j\zeta_j\in \mathbb{R} \langle \underline X \rangle$ be a noncommutative polynomial, without any constant term, where $h$ is a SOHS.  Assume that $\rc(\eta,i)\neq \zeta_j$, $1\leq j \leq r$, for any $i\geq 0$ and for any monomial $\eta$ of $h$. Then
\begin{enumerate}
\item It is always possible to obtain a Gram-like matrix preserving extension $\widetilde{f}$ of $f$ w.r.t $h$ by adding at most $2r+1$ terms with $f$.
\item  In fact, infinitely many such $\widetilde{f}$ can be obtained.
\end{enumerate}
\end{theorem}
\begin{proof}
Let $\deg(f)=d$.  By hypothesis, $f$ can be written as follows:
\begin{equation*}
f=h+\sum_{j=1}^r a_j\zeta_j=\sum_{i=1}^m g_i^{\ast}g_i+\sum_{j=1}^r a_j\zeta_j,
\end{equation*}
where $a\in \mathbb{R}$, $g_i$s and $\zeta$ lie in $\mathbb{R} \langle \underline X \rangle$, $\deg(g_i)\leq \tfrac{d}{2}$ if $d$ is even or $\deg(g_i)\leq \tfrac{d-1}{2}$ if $d$ is odd, for all $1\leq i \leq m$ and $1\leq \deg(\zeta_j)\leq d$ for all $1\leq j \leq r$.  Let $G_h$, say of order $k$, be a positive semidefinite Gram-like matrix of $h$ such that there exist a column vector $W_k$ satisfying
\begin{enumerate}[(a)]
\item $1$ does not appear as an entry of $W_k$,
\item $\zeta_1,\ldots, \zeta_r$ do not appear as entries of $W_k$, and
\item \begin{equation}\label{Gram matrix of the SOHS part of the given polynomial_with multiple terms in the non-SOHS part}
h=\sum_{i=1}^m g_i^{\ast}g_i=W_k^{\ast}\;G_h\;W_k.
\end{equation}
\end{enumerate}
It can be noted that existence of such $W_k$ satisfying
\begin{itemize}
    \item (a) is guaranteed by the assumption that $f$, hence $\sum_{i=1}^m g_i^{\ast}g_i$, doesn't contain any constant term,
    \item (b) is guaranteed by the hypothesis $\rc(\eta,i)\neq \zeta_j$, $1\leq j \leq r$, for any $i\geq 0$ and for any monomial $\eta$ of the SOHS part of $f$  (Indeed, let $\zeta_j$ appears as an entry of $W_k$ for some $j$.  Then $\zeta_j^{\ast}\zeta_j$ appears in the expansion of $\sum_{i=1}^m g_i^{\ast}g_i$ as its coefficient in nonzero by \eqref{subeqn : b} of Theorem \ref{Theorem_general form of a SOHS polynomial}.  But then $\rc(\zeta_j^{\ast}\zeta_j,\deg(\zeta_j))=\zeta_j$.),
    \item (c) is guaranteed by  Theorem \ref{Theorem_general form of a SOHS polynomial}.
\end{itemize}
Now we consider the matrix $\widetilde{G}$ given as follows:
\begin{equation}\label{construction of psd matrix corresponding to extended SOHS_with multiple terms in non-SOHS part}
\widetilde{G}:=
\left(\begin{array}{@{}c|c@{}}
  \begin{matrix}
  G_h
  \end{matrix}
  & \bigzero \\
\hline
  \bigzero &
  \begin{matrix}
  \sum_{j=1}^r d_i^{2} + 1 & d_1 & d_2 & \cdots & d_r\\
  d_1 & 1 & 0 & \cdots & 0\\
  d_2 & 0 & 1 & \cdots & 0\\
  \vdots & \vdots & \vdots & \ddots & \vdots\\
  d_r & 0 & 0 & \cdots & 1
  \end{matrix}
\end{array}\right)\;,
\end{equation} 
where
\begin{equation*}
d_i=\left\{ \begin{array}{ll}
a_i & \mbox{if $\zeta_i^{\ast}\neq \zeta_i$}\;;\\\\
\tfrac{a_i}{2} & \mbox{if $\zeta_i^{\ast}= \zeta_i$}\;.
\end{array} \right.
\end{equation*}
The lower diagonal block of the matrix $\widetilde{G}$, as in \eqref{construction of psd matrix corresponding to extended SOHS_with multiple terms in non-SOHS part}, is positive semidefinite by Lemma \ref{Lemma_psdness in terms of leading principal minor and Schur complement}.  Moreover, as the matrix $G_h$, as in \eqref{Gram matrix of the SOHS part of the given polynomial_with multiple terms in the non-SOHS part}, is positive semidefinite, so is the matrix $\widetilde{G}$ by Lemma \ref{Lemma_block psd(pd)implies psd(pd)}.\\
We now define  $\widetilde{W_{k+r+1}}$ as follows.
\begin{equation}\label{construction of a column vector of monomials corresponding to extended SOHS_with multiple terms in non-SOHS part}
\widetilde{W_{k+r+1}}:=
\left(\begin{array}{@{}c}
  \begin{matrix}
  W_k^t & 1 & \zeta_1 & \cdots & \zeta_r
  \end{matrix}
\end{array}\right)^t\;.
\end{equation}
We define $\widetilde{f}$ as follows :
\begin{equation}\label{construction of SOHS extension_non-SOHS part having multiple terms}
\widetilde{f}:=\widetilde{W_{k+r+1}}^{\ast}\;\widetilde{G}\;\widetilde{W_{k+r+1}},
\end{equation} 
where $\widetilde{G}$ and $\widetilde{W}_{k+r+1}$ are as in \eqref{construction of psd matrix corresponding to extended SOHS_with multiple terms in non-SOHS part} and \eqref{construction of a column vector of monomials corresponding to extended SOHS_with multiple terms in non-SOHS part} respectively.  Then we have 
\begin{equation}\label{showing the constructed SOHS as an extension_non-SOHS part having multiple terms}
\begin{split}
\widetilde{f}&=\sum_{i=1}^m g_i^{\ast}g_i+\sum_{j=1}^r d_j\zeta_j+\sum_{j=1}^r d_j\zeta_j^{\ast}+\sum_{j=1}^r \zeta_j^{\ast}\zeta_j+\sum_{j=1}^r d_i^{2} + 1\\
&=f+\sum_{\zeta_j^{\ast}\neq \zeta_j} a_j\zeta_j^{\ast}+\sum_{j=1}^r \zeta_j^{\ast}\zeta_j+\sum_{j=1}^r d_i^{2} + 1.
\end{split}
\end{equation}
It can be noted that $\widetilde{f}$ is obtained by adding at most $2r+1$ many terms with $f$. Therefore, first part of the assertion now follows from Theorem \ref{Theorem_general form of a SOHS polynomial}, \eqref{construction of SOHS extension_non-SOHS part having multiple terms} and \eqref{showing the constructed SOHS as an extension_non-SOHS part having multiple terms}.

It can be observed that any $\widetilde{G}$ obtained by replacing the $(k+1,k+1)$-th entry of the matrix by any $A\geq\sum_{j=1}^r d_i^{2}$ will be positive semidefinite by Lemma \ref{Lemma_psdness in terms of leading principal minor and Schur complement} and hence will produce another $\widetilde{f}$ satisfying the first part of the assertion.  Hence, we have infinitely many choices for such $\widetilde{f}$. 
\end{proof}
\subsection{Positive semidefinite completion problem and SOHS via graphs}
\label{subsec : Positive semidefinite completion problem and SOHS via graphs}
Positive semidefinite completion problem is one of the most important problems not only in the realm of mathematics but also in physics.  Let $A$ be a $n\times n$ partially specified positive semidefinite matrix, i.e, a matrix with some unknown entries and some natural conditions on the known entries.  Positive semidefinite completion problem deals with finding the pattern of known entries (and hence of the unknown entries as well) for which $A$ can be completed to a positive semidefinite matrix.  Apart from having huge applications, this problem is very intriguing on its own.  For example, a solution of this problem has been given via graph theory. 
 It is known that such a solution exists when the pattern of known entries is given by a chordal graph.       

In this subsection, we answer Question \ref{Introduction_main question_possible modification} via positive semidefinite completion problem.  We recall some basic definitions and results related to this and study the same problem from the context of extending a partial SOHS polynomial to a SOHS polynomial.  
\begin{definition}\label{partial symmetric matrix_definition}
A partially specified $n\times n$ matrix is a \textit{partial symmetric matrix} if 
\begin{itemize}
\item all the diagonal entries are specified, and
\item if the $(i,j)$-th entry is specified, then so is the $(j,i)$-th entry and they are equal.
\end{itemize} 
\end{definition}

\begin{definition}\label{partial positive (semi) definite matrix_definition}
A partial symmetric matrix is said to be a \textit{partial positive definite matrix} (respectively \textit{partial positive semidefinite matrix}) if any of its completely specified principal minor is positive (respectively nonnegative). 
\end{definition}
\begin{definition}\label{definition_specification graph}
The \textit{specification graph} $\mathcal{G}_A$ of a $n\times n$ partial symmetric matrix $A$ is a graph with vertex set $V=\{1,2,\ldots,n\}$ and for $i,j\in V$ with $i\neq j$, the vertices $i$ and $j$ are adjacent whenever $(i,j)$-th entry of $A$ is specified. 
\end{definition}
\begin{definition}\label{definition_completable}
A graph $\mathcal{G}$ with vertex set $\{1,2,\ldots,n\}$ is \textit{positive definite completable} (respectively \textit{positive semidefinite completable}) if any $n\times n$ partial positive definite (respectively semidefinite) matrix with specification graph $\mathcal{G}$ is completable to a positive definite (respectively semidefinite) matrix. 
\end{definition}

\begin{definition}\label{chordal graph_definition}
A graph is said to be \textit{chordal} if it doesn't contain the cycle $\mathcal{C}_k$ on $k$ vertices as an induced subgraph for any $k\geq 4$.
\end{definition}
\begin{example}
\begin{enumerate}
\item The complete graph $\mathcal{K}_n$ on $n$ vertices, trees are some examples of chordal graphs.
\item The graph $\mathcal{C}_4$, the cycle on $4$ vertices, is not a chordal graph.
\end{enumerate}
\end{example}
\begin{theorem}\label{proposition_completable iff chordal specification graph}
A graph $\mathcal{G}$ is positive definite (semidefinite) completable if and only if $\mathcal{G}$ is chordal.
\end{theorem}
\begin{proof}
See \cite[Theorem 7, p. 120]{GJSW} and \cite[Lemma 12.7, p.~161 and Theorem 12.11, p.~163]{Ba}.
\end{proof}
Now we describe what one means by incompleteness of a symmetric noncommutative polynomial w.r.t its partially specified Gram-like matrix and its completion. 
\begin{definition}\label{definition_incomplete polynomial}
 A symmetric noncommutative polynomial $f$ of degree $2d$ is said to be \textit{incomplete w.r.t its representation} $\mathbf{R}_f$, as in \eqref{equation_ most general form of a representation of a symmetric polynomial with positive co-efficients of square terms}, along with conditions as in \eqref{conditions on the coefficients of general form of a representation of a symmetric polynomial with positive co-efficients of square terms}, if
\begin{itemize}
\item  all $a_i$s are specified (positive real numbers by conditions as in \eqref{subeqn : b'}), and
\item some $b_{ij}$s are unspecified.
\end{itemize}
\end{definition}

\begin{definition}\label{definition_completion of a polynomial}
Let $W_{k}$ be a column vector of size $k$ of all  monomials from $\mathbb{R}\langle \underline{X}\rangle$ of degree at most $d$ taken in any arbitrary order and $G$ a $k\times k$ partial positive semidefinite matrix with positive diagonal entries.  Then for any positive semidefinite completion $\overline{G}$ of $G$, the polynomial $W_{k}^{\ast}\;\overline{G}\;W_{k}$ is called \textit{a SOHS completion of the polynomial} $f:=W_{s(d,n)}^{\ast}\;G\;W_{s(d,n)}$ \textit{w.r.t partial positive semidefinite Gram-like matrix} $G$, and is denoted by $\overline{f}^{\;G,\;W_k}$.
\end{definition}
We now want to look at incomplete polynomials that are not SOHS, but potentially very close to SOHS polynomials.  That is, whenever we collect appropriate monomials (with known coefficients) of such an incomplete polynomial $f$ which have the potential to form a SOHS, it does so.  
\begin{definition}\label{Definition_quasi SOHS}
Let $f$ be an incomplete polynomial w.r.t a given representation $\mathbf{R}_f$, as in \eqref{equation_ most general form of a representation of a symmetric polynomial with positive co-efficients of square terms}, along with conditions as in \eqref{conditions on the coefficients of general form of a representation of a symmetric polynomial with positive co-efficients of square terms}.  Then $f$ is said to be \textit{quasi SOHS w.r.t} $\mathbf{R}_f$ if for any $K\subset \{1,2,\ldots, k\}$ such that $b_{ij}$s are known for all $i,j\in K$,
\begin{equation*}
\sum_{i\in K}a_i\zeta_i^{\ast}\zeta_i+\sum_{i,j\in K, i < j} b_{ij}(\zeta_i^{\ast}\zeta_j+\zeta_j^{\ast}\zeta_i)
\end{equation*}
is a SOHS.
\end{definition}
We are now in a stage to reformulate Theorem \ref{proposition_completable iff chordal specification graph} in the context of the dictum of this paper.  To be specific, we provide a criteria regarding when a quasi SOHS polynomial qualifies to a SOHS polynomial.  We have the following theorem in that regard. 
\begin{theorem}\label{main theorem_SOHS completion of a quasi-SOHS iff known pattern is chordal}
Let $f\in \Sym\;\mathbb{R}\langle \underline{X}\rangle$ be a quasi SOHS polynomial w.r.t a representation $\mathbf{R}_f$, as in
\eqref{equation_ most general form of a representation of a symmetric polynomial with positive co-efficients of square terms}, along with conditions as in \eqref{conditions on the coefficients of general form of a representation of a symmetric polynomial with positive co-efficients of square terms}.  Then any such $f$ has a SOHS completion $\overline{f}^{\;G_f,\;W_k}$ w.r.t the partial positive semidefinite Gram-like matrix $G_f$ associated with $\mathbf{R}_f$ if and only if the specification graph of $G_f$ is chordal.  Here $W_k$ is the column vector $\begin{pmatrix}
\zeta_1 & \cdots & \zeta_k
\end{pmatrix}^{t}$, $\zeta_i$, $1\leq i \leq k$, being the monomials appearing in $\mathbf{R}_f$.
\end{theorem}
\begin{proof}
Let $f$ be a quasi SOHS polynomial of degree $2d$ w.r.t a representation 
\begin{equation*}
\mathbf{R}_f \;\;:\;\; f=\sum_{i=1}^{k}a_i\zeta_i^{\ast}\zeta_i+\sum_{1\leq i < j\leq k} b_{ij}(\zeta_i^{\ast}\zeta_j+\zeta_j^{\ast}\zeta_i),
\end{equation*}
where $a_i,b_{ij}\in \mathbb{R}$, $\zeta_i\in \mathbb{R}\langle \underline{X}\rangle$ are all monomials of degree at most $d$.  Therefore, for any $K\subset \{1,2,\ldots, k\}$ such that $b_{ij}$s are known for all $i,j\in K$,
\begin{equation*}
\sum_{i\in K}a_i\zeta_i^{\ast}\zeta_i+\sum_{i,j\in K, i < j} b_{ij}(\zeta_i^{\ast}\zeta_j+\zeta_j^{\ast}\zeta_i)
\end{equation*}
is a SOHS.  Hence, for any such set $K=\{i_0,i_1,\ldots, i_{\mid K \mid}\}$ with $i_0<i_1<\cdots <i_{\mid K \mid}$, the matrix
\begin{equation*}
\begin{pmatrix}
 a_{i_1} & b_{i_1i_2} & \cdots & b_{i_1i_{\mid K \mid}}\\
 b_{i_1i_2} & a_{i_2} & \cdots  & b_{i_2i_{\mid K \mid}}\\
  \vdots & \vdots & \ddots & \vdots\\
  b_{i_1i_{\mid K \mid}} & b_{i_2i_{\mid K \mid}}  & \cdots & a_{i_{\mid K \mid}}
\end{pmatrix}
\end{equation*}
is positive semidefinite by Theorem \ref{Theorem_general form of a SOHS polynomial}.  Therefore any specified principal minor of the Gram-like matrix $G_f$, associated with $\mathbf{R}_f$ is positive semidefinite, that is to say $G_f$ is a partial positive semidefinite matrix.  Therefore, it makes sense to talk about SOHS completion of $f$ w.r.t $G_f$ following Definition \ref{definition_completion of a polynomial}.  Now, the assertion follows from Theorem \ref{proposition_completable iff chordal specification graph}.
\end{proof}
\begin{remark}\label{remark_two examples, completable and not completable}
\begin{enumerate}
\item It is easy to observe that Theorem \ref{main theorem_SOHS completion of a quasi-SOHS iff known pattern is chordal} holds for Gram matrix as well, by taking $k=s(d,n)$.  To do so, one only has to analogously define Definition \ref{definition_incomplete polynomial}, Definition \ref{definition_completion of a polynomial} and Definition \ref{Definition_quasi SOHS}  w.r.t a representation as $\mathbf{R}_f$, as in \eqref{equation_general form of a symmetric polynomial}, along with conditions in \eqref{conditions on the coefficients of general form of a representation of a symmetric polynomial}.
\item If we take the specification graph of $G_f$ as the complete graphs $\mathcal{K}_{s(d,n)}$ and $\mathcal{K}_{k}$, then we get back Theorem \ref{Theorem_SOHS iff psd Gram matrix} and Theorem \ref{Theorem_general form of a SOHS polynomial} respectively.  That is to say, Theorem \ref{main theorem_SOHS completion of a quasi-SOHS iff known pattern is chordal} generalises Theorem \ref{Theorem_SOHS iff psd Gram matrix} and Theorem \ref{Theorem_general form of a SOHS polynomial}.
\item Let $\mathcal{G}$ be a nonchordal graph.  Then also there might exist a quasi SOHS polynomial $f$ with $\mathcal{G}$ as the specification graph of its Gram matrix $G_f$ which has a SOHS completion.  For example, let $\zeta_1=1$, $\zeta_2=X_1X_2$, $\zeta_3=X_2X_1$ and $\zeta_4=X_2^2$ and $W_4=\begin{pmatrix}
1 & X_1X_2 & X_2X_1 & X_2^2
\end{pmatrix}^{t}$.  Consider the following polynomial $f$ given as
\begin{equation}\label{equation_apparent counter example of completable iff chordal}
\begin{split}
f=&5\zeta_1^{\ast}\zeta_1+5\zeta_2^{\ast}\zeta_2+5\zeta_3^{\ast}\zeta_3++5\zeta_4^{\ast}\zeta_4+5(\zeta_1^{\ast}\zeta_2+\zeta_2^{\ast}\zeta_1)+\ast(\zeta_1^{\ast}\zeta_3+\zeta_3^{\ast}\zeta_1)\\
&+5(\zeta_1^{\ast}\zeta_4+\zeta_4^{\ast}\zeta_1)+5(\zeta_2^{\ast}\zeta_3+\zeta_3^{\ast}\zeta_2)+\ast(\zeta_2^{\ast}\zeta_4+\zeta_4^{\ast}\zeta_2)+5(\zeta_3^{\ast}\zeta_4+\zeta_4^{\ast}\zeta_3)\\
=5&+(5+\ast)X_1X_2+(5+\ast)X_2X_1+10X_2^2+5X_1X_2^2X_1+5X_2X_1^2X_2+5X_1X_2X_1X_2\\
&+5X_2X_1X_2X_1
+5X_1X_2^3+5X_2^3X_1+\ast X_2X_1X_2^2+\ast X_2^2X_1X_2+5X_2^4.
\end{split}
\end{equation}
It is easy to check that $f$ is a quasi SOHS polynomial of degree $4$ w.r.t the given representation and the matrix $G_f$ given by 
\begin{equation*}\label{equation_partial psd matrix with unique psd completion}
G_f=\begin{pmatrix}
5 & 5 & \ast & 5\\
5 & 5 & 5 & \ast\\
\ast & 5 & 5 & 5\\
5 & \ast & 5  & 5\\
\end{pmatrix}
\end{equation*}    
is a partial positive semidefinite Gram-like matrix of $f$.  The matrix $G_f$ is also a partial positive semidefinite matrix as $f$ is quasi SOHS and by Theorem \ref{Theorem_general form of a SOHS polynomial}.  It can also be noted that the specification graph of $G_f$ is $\mathcal{C}_4$, which is not chordal.  Still there exists a SOHS completion $\overline{f}^{\;G_f,\;W_4}$, which is obtained by taking $\ast=5$ in \eqref{equation_apparent counter example of completable iff chordal} as the positive semidefinite matrix 
\begin{equation*}
\overline{G_f}=\begin{pmatrix}
5 & 5 & 5 & 5\\
5 & 5 & 5 & 5\\
5 & 5 & 5 & 5\\
5 & 5 & 5 & 5\\
\end{pmatrix}
\end{equation*}
is clearly a positive semidefinite completion of $G_f$.  But, this doesn't contradict Theorem \ref{main theorem_SOHS completion of a quasi-SOHS iff known pattern is chordal}.  

To see this, we first note that determinant of the matrix 
\begin{equation}\label{equation_main matrix required to prove cycle of length four is not completable}
\begin{pmatrix}
5 & 5 & \ast\\
5 & 5 & 5\\
\ast & 5 & 5
\end{pmatrix}
\end{equation}   
is $-(\ast-5)^2$.  Therefore, the matrix as in \eqref{equation_main matrix required to prove cycle of length four is not completable} is positive semidefinite if and only if $\ast=5$.  That is to say, the only positive semidefinite completion of this partial positive semidefinite matrix is
\begin{equation*} 
\begin{pmatrix}
5 & 5 & 5\\
5 & 5 & 5\\
5 & 5 & 5
\end{pmatrix}\;.
\end{equation*}     
We now exploit this uniqueness.  Consider the polynomial
\begin{equation}\label{equation_apparent counter example of completable iff chordal is not at all so}
\begin{split}
f^{\prime}=&5\zeta_1^{\ast}\zeta_1+5\zeta_2^{\ast}\zeta_2+5\zeta_3^{\ast}\zeta_3+5\zeta_4^{\ast}\zeta_4+5(\zeta_1^{\ast}\zeta_2+\zeta_2^{\ast}\zeta_1)+\ast(\zeta_1^{\ast}\zeta_3+\zeta_3^{\ast}\zeta_1)\\
&+\sqrt{5}(\zeta_1^{\ast}\zeta_4+\zeta_4^{\ast}\zeta_1)+5(\zeta_2^{\ast}\zeta_3+\zeta_3^{\ast}\zeta_2)+\ast(\zeta_2^{\ast}\zeta_4+\zeta_4^{\ast}\zeta_2)+5(\zeta_3^{\ast}\zeta_4+\zeta_4^{\ast}\zeta_3)\\
=5&+(5+\ast)X_1X_2+(5+\ast)X_2X_1+2\sqrt{5}X_2^2+5X_1X_2^2X_1+5X_2X_1^2X_2+5X_1X_2X_1X_2\\
&+5X_2X_1X_2X_1
+5X_1X_2^3+5X_2^3X_1+\ast X_2X_1X_2^2+\ast X_2^2X_1X_2+5X_2^4.
\end{split}
\end{equation}
Again it is a qusi SOHS polynomial of degree $4$ and associated partial positive semidefinite Gram-like matrix $G_{f^{\prime}}$ is given by 
\begin{equation}\label{equation_not completable}
G_{f^{\prime}}=\begin{pmatrix}
5 & 5 & \ast & \sqrt{5}\\
5 & 5 & 5 & \ast\\
\ast & 5 & 5 & 5\\
\sqrt{5} & \ast & 5  & 5\\
\end{pmatrix}
\end{equation}
having specification graph $\mathcal{C}_4$.  Suppose, the matrix $G_{f^{\prime}}$ is positive semidefinite completable.  Then first consider the leading principal submatrix of $G_{f^{\prime}}$ obtained by deleting last row and column.  Then, as this has to be positive semidefinite after completion, $\ast$ in $(1,3)$-th and $(3,1)$-th position must have to be $5$ by the aforementioned uniqueness.  Now, considering the principal submatrix obtained by deleting first row and column and applying same argument, we obtain $\ast$ in $(2,4)$-th and $(4,2)$-th position must have to be $5$.  But then the principal submatrix of $G_{f^{\prime}}$, obtained either by deleting second row and second column (or by deleting third row and third column), has determinant $(\sqrt{5}-5)(25-5\sqrt{5})(<0)$.  Therefore, the matrix $G_{f^{\prime}}$ is not  positive semidefinite completable and hence the polynomial, as in \eqref{equation_apparent counter example of completable iff chordal is not at all so} doesn't have a SOHS completion. 
\end{enumerate}     
\end{remark}
We now go through some applications of Theorem \ref{main theorem_SOHS completion of a quasi-SOHS iff known pattern is chordal}.
\begin{proposition}\label{Application of main theorem involving chordal graph}
Let $f$ be a quasi SOHS polynomial of degree $2d$ w.r.t to the representation $\mathbf{R}_f$
\begin{equation*}
\mathbf{R}_f\; :\; f=\sum_{i=1}^{k}a_i\zeta_i^{\ast}\zeta_i+\sum_{1\leq i < j\leq k} b_{ij}(\zeta_i^{\ast}\zeta_j+\zeta_j^{\ast}\zeta_i),
\end{equation*}
as in \eqref{equation_ most general form of a representation of a symmetric polynomial with positive co-efficients of square terms} along with conditions as in \eqref{conditions on the coefficients of general form of a representation of a symmetric polynomial with positive co-efficients of square terms}.  If out of $\tfrac{k(k-1)}{2}$ many $b_{ij}$s, any $k-1$ are specified and remaining $\tfrac{(k-1)(k-2)}{2}$ are unspecified, and if there is no column (or row) of the partial positive semidefinite Gram-like matrix $G_f$ associated with $\mathbf{R}_f$ with diagonal entry as the only specified entry, then $f$ has a SOHS completion.  
\end{proposition}
\begin{proof}
By hypothesis, in any column of $G_f$ other than the diagonal entry at least one entry is specified.  This ensures that the specification graph $\mathcal{G}_{G_f}$ of $G_f$ is connected.  Moreover, the hypothesis that any $k-1$ many $b_{ij}$s are specified affirms that there are exactly $k-1$ edges between $k$ many vertices of $\mathcal{G}_{G_f}$.  Therefore, the graph $\mathcal{G}_{G_f}$ is indeed a tree and hence chordal.  Hence, the assertion follows from Theorem \ref{main theorem_SOHS completion of a quasi-SOHS iff known pattern is chordal}.       
\end{proof}
\begin{corollary}\label{corollary_rooted tree of height one}
Let $f$ be a quasi SOHS polynomial of degree $2d$ w.r.t to the representation $\mathbf{R}_f$
\begin{equation*}
\mathbf{R}_f\; :\; f=\sum_{i=1}^{k}a_i\zeta_i^{\ast}\zeta_i+\sum_{1\leq i < j\leq k} b_{ij}(\zeta_i^{\ast}\zeta_j+\zeta_j^{\ast}\zeta_i),
\end{equation*}
as in \eqref{equation_ most general form of a representation of a symmetric polynomial with positive co-efficients of square terms} along with conditions as in \eqref{conditions on the coefficients of general form of a representation of a symmetric polynomial with positive co-efficients of square terms}.  If out of $\tfrac{k(k-1)}{2}$ many $b_{ij}$s, exactly $k-1$ are specified for some fixed $i$ (or fixed $j$), then $f$ has a SOHS completion.
\end{corollary}
\begin{proof}
Follows from Proposition \ref{Application of main theorem involving chordal graph}.  
\end{proof}
\begin{remark}
It can be noted that, taking $k=4$ in Corollary \ref{corollary_rooted tree of height one}, the specification graphs $\mathcal{G}_{G_f}$ are the following rooted trees of height one on $4$ vertices.
\end{remark}
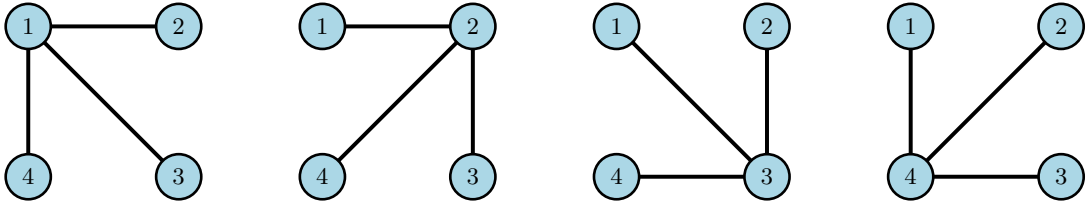
\begin{figure}[!h]
\hspace{.3cm}
{\begin{tikzpicture}
\Vertex[y=2,label=$1$]{A}
\Vertex[x=2,y=2,label=$2$]{B}
\Vertex[x=2,label=$3$]{C}
\Vertex[label=$4$]{D}
\Edge[color=black](A)(B)
\Edge[color=black](A)(D)
\Edge[color=black](A)(C)
\end{tikzpicture}} \hspace{1cm}
{\begin{tikzpicture}
\Vertex[y=2,label=$1$]{A}
\Vertex[x=2,y=2,label=$2$]{B}
\Vertex[x=2,label=$3$]{C}
\Vertex[label=$4$]{D}
\Edge[color=black](A)(B)
\Edge[color=black](B)(D)
\Edge[color=black](B)(C)
\end{tikzpicture}}
\hspace{1cm}
{\begin{tikzpicture}
\Vertex[y=2,label=$1$]{A}
\Vertex[x=2,y=2,label=$2$]{B}
\Vertex[x=2,label=$3$]{C}
\Vertex[label=$4$]{D}
\Edge[color=black](A)(C)
\Edge[color=black](B)(C)
\Edge[color=black](C)(D)
\end{tikzpicture}}
\hspace{1cm}
{\begin{tikzpicture}
\Vertex[y=2,label=$1$]{A}
\Vertex[x=2,y=2,label=$2$]{B}
\Vertex[x=2,label=$3$]{C}
\Vertex[label=$4$]{D}
\Edge[color=black](A)(D)
\Edge[color=black](B)(D)
\Edge[color=black](C)(D)
\end{tikzpicture}}
\end{figure}

\begin{corollary}
 Let $f$ be a quasi SOHS polynomial of degree $2d$ w.r.t to the representation $\mathbf{R}_f$
\begin{equation*}
\mathbf{R}_f\; :\; f=\sum_{i=1}^{k}a_i\zeta_i^{\ast}\zeta_i+\sum_{1\leq i < j\leq k} b_{ij}(\zeta_i^{\ast}\zeta_j+\zeta_j^{\ast}\zeta_i),
\end{equation*}
as in \eqref{equation_ most general form of a representation of a symmetric polynomial with positive co-efficients of square terms} along with conditions as in \eqref{conditions on the coefficients of general form of a representation of a symmetric polynomial with positive co-efficients of square terms}.  If out of $\tfrac{k(k-1)}{2}$ many $b_{ij}$s, at most $3$ are specified, then $f$ has a SOHS completion.   
\end{corollary}
\begin{proof}
    The hypothesis ensures that the specification graph $\mathcal{G}_{G_f}$ of the partial positive semidefinite Gram-like matrix $G_f$ associated with $\mathbf{R}_f$ has $k$ vertices and at most $3$ edges.  Therefore, $\mathcal{G}_{G_f}$ can't have $\mathcal{C}_n$ as an induced subgraph for $n\geq 4$ and $\mathcal{G}_{G_f}$ is chordal.  Therefore it follows from Theorem \ref{main theorem_SOHS completion of a quasi-SOHS iff known pattern is chordal}.   
\end{proof}

\subsection{Positive semidefinite completion problem and SOHS via algebraic sets}
\label{subsec : Positive semidefinite completion problem and SOHS via varieties}
In Subsection \ref{subsec : Positive semidefinite completion problem and SOHS via graphs}, we answered Question \ref{Introduction_main question_possible modification} using techniques of solving positive semidefinite completion problem via graph theory.  Recall that, given a partial positive semidefinite matrix $G$, we considered a \textit{graph} associated with it which was obtained by keeping track of the \textit{specified off diagonal entries} of $G$.  On the contrary, now we answer Question \ref{Introduction_main question_possible modification} using another method of solving positive semidefinite completion problem.  We do so by associating a \textit{projective algebraic set} to it that keeps track of the \textit{unspecified off diagonal entries}.  In the process, we obtain a nice connection between noncommutative polynomial and usual (commutative) polynomial.  

In this section, by a polynomial we mean the usual commutative polynomial.  We denote commutative variables by small alphabets, unlike the noncommutative variables.  By a form we mean a homogeneous polynomial.  It is obvious that if a (homogeneous) polynomial has a sum of square (SOS) representation then it should be nonnegative.  But the converse is not true.  In fact, Hilbert, in 1888, proved that converse is true only for quadratic forms, binary forms and ternary quartics (cf. \cite{H0}).  Many years later, Motzkin gave a specific example of ternary sextic, namely $x_0^6+x_1^4x_2^2+x_1^2x_2^4-3x_0^2x_1^2x_2^2$, which is a nonnegative polynomial but can not be expressed as a SOS.  In 2016, Blekherman, Smith and Velasco has generalised Hilbert's result.  Instead of looking into globally nonnegative polynomials, they looked at nonnegative quadratic forms over a nondegenerate totally-real projective variety $X\subset \mathbb{P}^n$.  Here, by a totally-real variety they meant that the set of real points in $X$ is a dense subset of set of complex points of $X$. They proved that any nonnegative polynomial would be a sum of squares, modulo the ideal of $X$ if and only if $X$ is a variety of minimal degree, i.e, $\deg(X)=\codim(X)+1$ (cf. \cite[Theorem 1.1,p. 893]{BSV1}).  In 2017, Blekherman, Sinn and Velasco generalised it even further for any nondegenerate totally-real projective algebraic set.  As the notion of degree doesn't make sense for reducible algebraic sets, they observed that a right generalisation of the notion of degree is Castelnuovo-Mumford regularity.  We now recall the meaning of this notion for a projective algebraic set.  
\begin{definition}
Let $k$ be a field of characteristic $0$.  A finitely generated $k[x_0,x_1,\ldots,x_n]$-module $M$ has a unique free resolution 
\begin{equation*}
\cdots \rightarrow F_t \rightarrow F_{t-1} \rightarrow \cdots \rightarrow F_0 \rightarrow M \rightarrow 0.
\end{equation*}
Let $b_j$ be the largest degree of a minimal generator in the module $F_j$.  We also fix the convention that $b_j=-\infty$ if $F_j$ is the zero module.  Then the module $M$ is said to be \text{$m$-regular} if $b_j\leq m+j$ for all $j$.   
\end{definition}
\begin{definition}
A projective algebraic set $X\subset \mathbb{P}^n$ is said to be $m$-regular if the defining ideal $I$ of $X$ is $m$-regular as a $k[x_0,x_1,\ldots,x_n]$-module.
\end{definition}
We now state the main result obtained by Blekherman, Sinn and Velasco generalising \cite[Theorem 1.1,p. 893]{BSV1}.
\begin{theorem}\label{non-negative equals sum of sqares iff 2-regular}
Let $X\subset \mathbb{P}^n$ be a totally-real projective algebraic set.  Every nonnegative quadratic form on $X$ is a sum of squares, modulo the ideal of $X$ if and only if $X$ is 2-regular.   
\end{theorem}
\begin{proof}
See \cite[Theorem 9, p. 181]{BSV2}.
\end{proof}
\begin{remark}
It can be noted that linear projective varieties are $1$-regular.  Apart from that, any varieties of minimal degree are $2$-regular.  That is why, the apt generalisation of the notion of degree for a reduced algebraic set is regularity. 
\end{remark} 
Apparently, it might seem that the results proved in \cite{BSV1} and \cite{BSV2} are only for quadratic forms, whereas Hilbert proved the same for any $2d$ degree polynomials.  But, they showed that it is enough to consider quadratic forms as any other even degree forms can be settled down using Veronese embedding.  To be precise, to look at degree $2d$ forms on a variety $X$, it is enough to consider quadratic forms on the image $\nu_d(X)$ of $X$ via $d$-th Veronese embedding $\nu_d$, (cf. \cite[p.~737]{BSSV}).

Both of the results, i.e, Theorem 1.1 in \cite{BSV1} and Theorem 9 in \cite{BSV2}, have a deep connection with positive semidefinite completion problem.  Recall that, given a partial positive semidefinite matrix $G$, its specification graph encodes the known part of $G$.  On the contrary, to encode the unknown part of $G$ a projective algebraic set can be associated with it.  In this regard, we have the following definition.  
\begin{definition}
The \textit{subspace arrangement} $V_A$ of a $n \times n$ partial symmetric matrix $A$ is a projective algebraic subset of $\mathbb{P}^{n-1}$  whose vanishing homogeneous ideal is generated by the (square free) quadratic monomials $x_{i-1}x_{j-1}$, $1\leq i <j \leq n$, whenever $(i,j)$-th entry of $A$ is unspecified. 
\end{definition}
As mentioned earlier, it is expected that given a partial symmetric matrix, its specification graph and subspace arrangement should be related.  Following result by Fr\"oberg points out that relation.  In the process, it provides a bridge between algebraic geometry, commutative algebra, graph theory and positive semidefinite completion problem.
\begin{theorem}\label{2-regular iff chordal}
Let $G$ be a partial symmetric matrix.  Then the subspace arrangement $V_G$ is $2$-regular if and only if the specification graph $\mathcal{G}_G$ of $G$ is chordal.
\end{theorem}
\begin{proof}
See \cite{F}.
\end{proof}

Unlike commutative setup, the relation between positive noncommutative polynomials and SOHS polynomials is simple.  In fact, any positive noncommutative polynomial is a SOHS polynomial (cf. Remark \ref{rem_Heltonresult}).  But the equality between the nonnegative quadratic forms and SOS modulo a suitable projective algebraic set has the potential to allow a SOHS extension of a quasi SOHS polynomial of any even degree.  We encode this in the following theorem, which is nothing but a reinterpretation of Theorem \ref{main theorem_SOHS completion of a quasi-SOHS iff known pattern is chordal} through the lens of algebraic geometry and commutative algebra.  Before that, we quickly recall a couple of definitions along with their notations.

We denote by $S$ the homogeneous coordinate ring of a projective algebraic set $X\subset \mathbb{P}^n$.  That is, if $I$ is the homogeneous ideal of $\mathbb{R}[x_0,\ldots,x_n]$ generated by all forms that vanish on $X$, then $S:=\tfrac{\mathbb{R}[x_0,\ldots,x_n]}{I}$.  By $S_j$ we denote the vector space over $\mathbb{R}$ spanned by the degree $j$ forms of $S$. Then by $P_X$ and $\Sigma_X$ we denote the following two sets.
\begin{equation*}
\begin{split}
P_X&:=\Big\{f\in S_2\;\mid\;f\geq 0 \text{\;on\;all\;real\;points\;} x\in X\Big\},\\
\Sigma_X&:=\Big\{\sum_i h_i^2\;\mid\;h_i\in S_1 \text{\;for\;all\;}i\Big\}.
\end{split}
\end{equation*}   
\begin{theorem}\label{main theorem_SOHS completion of a quasi-SOHS iff unknown pattern is 2-regular}
Let $f\in \Sym\;\mathbb{R}\langle \underline{X}\rangle$ be a quasi SOHS polynomial w.r.t a representation $\mathbf{R}_f$, as in
\eqref{equation_ most general form of a representation of a symmetric polynomial with positive co-efficients of square terms}, along with conditions as in \eqref{conditions on the coefficients of general form of a representation of a symmetric polynomial with positive co-efficients of square terms}.  Then the following are equivalent:
\begin{enumerate}
\item Any such $f$ has a SOHS completion $\overline{f}^{\;G_f,\;W_k}$ w.r.t the partial positive semidefinite Gram-like matrix $G_f$ associated with $\mathbf{R}_f$.
\item The subspace arrangement $V_{G_f}(\subseteq \mathbb{P}^{k-1})$ of $G_f$ satisfies $P_{V_{G_f}}=\Sigma_{V_{G_f}}$.
\item The subspace arrangement $V_{G_f}(\subseteq \mathbb{P}^{k-1})$ of $G_f$ is $2$-regular.
\end{enumerate}
Here $W_k$ is the column vector $\begin{pmatrix}
\zeta_1 & \cdots & \zeta_k
\end{pmatrix}^{t}$, $\zeta_i$, $1\leq i \leq k$, being the monomials appearing in $\mathbf{R}_f$.
\end{theorem}
\begin{proof}
Follows from Theorem \ref{main theorem_SOHS completion of a quasi-SOHS iff known pattern is chordal}, Theorem \ref{non-negative equals sum of sqares iff 2-regular} and Theorem \ref{2-regular iff chordal}.
\end{proof}
We now apply Theorem \ref{main theorem_SOHS completion of a quasi-SOHS iff known pattern is chordal} and Theorem \ref{main theorem_SOHS completion of a quasi-SOHS iff unknown pattern is 2-regular} to provide a couple of examples of some quasi SOHS polynomials that admit SOHS completions.
\begin{example}
Consider any quasi SOHS polynomial $f$ whose partial positive semidefinite Gram-like matrix $G_f$ is given as follows :
\begin{equation*}
G_f=\begin{pmatrix}
5 & 5 & 5 & \ast\\
5 & 5 & 5 & \ast\\
5 & 5 & 5 & 5\\
\ast & \ast & 5  & 5\\
\end{pmatrix}.
\end{equation*}
Then the corresponding subspace arrangement $V_{G_f}$ is the projective algebraic set sitting inside $\mathbb{P}^3$ whose defining ideal is the homogeneous ideal $I:=\langle x_0x_3,x_1x_3 \rangle$.  Using Macaulay2 software (cf. \cite{GS}), we compute the free resolution and Betti table of $I$, considered as a $\mathbb{R}[x_0,x_1,x_2,x_3]$-module, as follows :

\textbf{Free resolution of} $I$ : Denoting $\mathbb{R}[x_0,x_1,x_2,x_3]$ by $R$, we have :\\ \begin{equation*}
\xymatrix{R^2&&\ar[ll]_{\begin{pmatrix}
-x_1\\
x_0
\end{pmatrix}} R^1&\ar[l]_{0} 0\\
0 && 1 & 2}
\end{equation*}

\textbf{Betti table of} $I$:\\
\begin{equation*}
\begin{split}
&0 \;\;\; 1\\
\text{total :}\; &2 \;\;\; 1\\
2:\;& 2 \;\;\; 1
\end{split}
\end{equation*}
Therefore, we obtain that the ideal $I$ is $2$-regular and so is the projective algebraic set $V_{G_f}$.  Therefore, by Theorem \ref{main theorem_SOHS completion of a quasi-SOHS iff unknown pattern is 2-regular}, any such quasi SOHS polynomial $f$ admits a SOHS completion w.r.t $G_f$.  This is also evident from Theorem \ref{main theorem_SOHS completion of a quasi-SOHS iff known pattern is chordal} as the specification graph $\mathcal{G}_{G_f}$ 
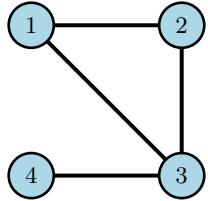
\begin{center}
\begin{tikzpicture}
\Vertex[y=2,label=$1$]{A}
\Vertex[x=2,y=2,label=$2$]{B}
\Vertex[x=2,label=$3$]{C}
\Vertex[label=$4$]{D}
\Edge[color=black](A)(B)
\Edge[color=black](A)(C)
\Edge[color=black](B)(C)
\Edge[color=black](C)(D)
\end{tikzpicture}
\end{center}
is chordal. 
\end{example}

\begin{example}
Consider any quasi SOHS polynomial $f$ whose partial positive semidefinite Gram-like matrix $G_f$ has the following disconnected specification graph $\mathcal{G}_f$ :
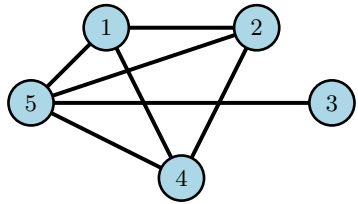
\begin{center}
\begin{tikzpicture}
\Vertex[x=1,y=2,label=$1$]{A}
\Vertex[x=3,y=2,label=$2$]{B}
\Vertex[x=4,y=1,label=$3$]{C}
\Vertex[x=2,label=$4$]{D}
\Vertex[y=1,label=$5$]{E}
\Edge[color=black](A)(B)
\Edge[color=black](C)(D)
\end{tikzpicture}
\end{center}
Then the corresponding subspace arrangement $V_{G_f}$ is the projective algebraic subset of $\mathbb{P}^4$ whose defining ideal is the homogeneous ideal $I:=\langle x_0x_2,x_0x_3, x_0x_4, x_1x_2,x_1x_3,x_1x_4,x_2x_4,x_3x_4\rangle$.  Using Macaulay2 software, we compute the free resolution and Betti table of $I$ considered as a $\mathbb{R}[x_0,x_1,x_2,x_3,x_4]$-module, as follows :

\textbf{Free resolution of} $I$ : Denoting $\mathbb{R}[x_0,x_1,x_2,x_3,x_4]$ by $R$, we have :\\ \begin{equation*}
\xymatrix{R^8&&&&\ar[llll]_{\begin{pmatrix}
-x_1 & \cdots & 0\\
0 & \cdots & 0\\
0 & \cdots & 0\\
x_0 & \cdots & 0\\
0 & \cdots & -x_4\\
0 & \cdots & x_3\\
0 & \cdots & 0\\
0 & \cdots & 0
\end{pmatrix}} R^{14}&&&\ar[lll]_{\begin{pmatrix}
x_4 & \cdots & 0\\
\vdots & \ddots & \vdots \\
0 & \cdots & x_2
\end{pmatrix}} R^9&&\ar[ll]_{\begin{pmatrix}
-x_3 & 0\\
\vdots & \vdots\\
0 & x_0
\end{pmatrix}} R^2&\ar[l]_{0} 0\\
0 &&&& 1 &&& 2 && 3 & 4}
\end{equation*}

\textbf{Betti table of} $I$:\\
\begin{equation*}
\begin{split}
&0 \;\;\;\; 1 \;\;\;\; 2 \;\;\; 3\\
\text{total :}\; &8 \;\;\; 14 \;\;\; 9 \;\;\; 2\\
2:\;&8 \;\;\; 14 \;\;\; 9 \;\;\; 2\\
\end{split}
\end{equation*}
Therefore, we obtain that the ideal $I$ is $2$-regular and so is the projective algebraic set $V_{G_f}$.  Therefore, by Theorem \ref{main theorem_SOHS completion of a quasi-SOHS iff unknown pattern is 2-regular}, any such quasi SOHS polynomial admits a completion w.r.t $G_f$.  In fact, it also follows from Theorem \ref{main theorem_SOHS completion of a quasi-SOHS iff known pattern is chordal} as the specification graph $\mathcal{G}_f$ is chordal.
\end{example}

We now apply Theorem \ref{main theorem_SOHS completion of a quasi-SOHS iff known pattern is chordal} and Theorem \ref{main theorem_SOHS completion of a quasi-SOHS iff unknown pattern is 2-regular} to provide an example of some quasi SOHS polynomials that do not admit SOHS completions.
\begin{example}
Consider the partial positive semidefinite matrix, as in \eqref{equation_not completable}, 
\begin{equation*}
G^{\prime}=\begin{pmatrix}
5 & 5 & \ast & \sqrt{5}\\
5 & 5 & 5 & \ast\\
\ast & 5 & 5 & 5\\
\sqrt{5} & \ast & 5  & 5\\
\end{pmatrix}.
\end{equation*}
Then the corresponding subspace arrangement $V_{G^{\prime}}$ is the projective algebraic set sitting inside $\mathbb{P}^3$ whose defining ideal is the homogeneous ideal $J:=\langle x_0x_2,x_1x_3 \rangle$.  Using Macaulay2 software (cf. \cite{GS}), we compute the free resolution and Betti table of $I$, considered as a $\mathbb{R}[x_0,x_1,x_2,x_3]$-module, as follows :

\textbf{Free resolution of} $J$ : Denoting $\mathbb{R}[x_0,x_1,x_2,x_3]$ by $R$, we have :\\ \begin{equation*}
\xymatrix{R^2&&\ar[ll]_{\begin{pmatrix}
-x_1x_3\\
x_0x_2
\end{pmatrix}} R^1&\ar[l]_{0} 0\\
0 && 1 & 2}
\end{equation*}

\textbf{Betti table of} $J$:\\
\begin{equation*}
\begin{split}
&0 \;\;\; 1\\
\text{total :}\; &2 \;\;\; 1\\
2:\;& 2 \;\;\; \textbf{.}\\
3:\;& \textbf{.} \;\;\; 1\\
\end{split}
\end{equation*}
Therefore, we obtain that the ideal $J$ is $3$-regular and so is the projective algebraic set $V_{G^{\prime}}$.  Therefore, by Theorem \ref{main theorem_SOHS completion of a quasi-SOHS iff unknown pattern is 2-regular}, there exists a quasi SOHS polynomial that doesn't admit a SOHS completion w.r.t $G^{\prime}$.  In fact, the polynomial $f^{\prime}$, as given in \eqref{equation_apparent counter example of completable iff chordal is not at all so}, is one such.  This is not all surprising.  In fact, it follows from Theorem \ref{main theorem_SOHS completion of a quasi-SOHS iff known pattern is chordal} as the specification graph $G^{\prime}$
\begin{center}
\begin{tikzpicture}
\Vertex[y=2,label=$1$]{A}
\Vertex[x=2,y=2,label=$2$]{B}
\Vertex[x=2,label=$3$]{C}
\Vertex[label=$4$]{D}
\Edge[color=black](A)(B)
\Edge[color=black](A)(D)
\Edge[color=black](B)(C)
\Edge[color=black](C)(D)
\end{tikzpicture}
\end{center}
is the cycle $\mathcal{C}_4$ on four vertices, which is not a chordal graph.  
\end{example}

\begin{remark}
The flexibility of a Gram-like matrix over a Gram matrix is quite evident from these examples.  As we deal with Gram-like matrices, we could play around with graphs having any number of vertices.  For a Gram matrix, we can only deal with graphs with $s(d,n)$ many vertices for some $d\geq 0$ and $n>0$.       
\end{remark}
\begin{remark}
    As a next step, it would be interesting to generate algorithms out of the criteria provided in Theorem \ref{Theorem_conditions satisfied my monomial vectors if if G-l MPE exists}, Theorem \ref{Theorem_NASC for block matrix form of a Gram-matrix of the extension}, Theorem \ref{Block PSD technique_with non-SOHS part having multiple terms}, Theorem \ref{main theorem_SOHS completion of a quasi-SOHS iff known pattern is chordal} and Theorem \ref{main theorem_SOHS completion of a quasi-SOHS iff unknown pattern is 2-regular} and use that further to produce SDP corresponding to these extension problems.
\end{remark}

\section*{Acknowledgements}
  Both the authors extend their gratitude to Prof. Sumesh K, Prof. Arijit Dey and Prof. Sriram B for their valuable suggestions regarding the article.  The authors acknowledge Rajib Sarkar and Sreenanda S B for their assistance regarding the use of Macaulay2 software.  The authors thank Dr. Aditi Howlader for suggesting some useful references.  The authors also thank Agnihotri Roy for reexamining some tedious calculations.  The first named author would like to thank Indian Institute of Technology Madras for financial support (Office order No.F.ARU/R10/IPDF/2024).
  The second named author would like to thank Indian Institute of Technology Madras and ICSR for financial support (Project number - SB22231267MAETWO008573).

\end{document}